\documentclass[letterpaper, 12 pt]{article}
\usepackage{amsthm, amsmath, amssymb, amsfonts, epsfig, psfrag, rotate}
\usepackage[matrix, arrow, curve]{xy}
\usepackage{comment}

\newcommand{\setleftmargin}[1]{
        \addtolength{\textwidth}{\oddsidemargin}
        \addtolength{\textwidth}{1in}
        \addtolength{\textwidth}{-#1}
        \setlength{\oddsidemargin}{-1in}
        \addtolength{\oddsidemargin}{#1}
        \setlength{\evensidemargin}{\oddsidemargin}
}
\newcommand{\setrightmargin}[1]{
        \setlength{\textwidth}{8.5in}
        \addtolength{\textwidth}{-\oddsidemargin}
        \addtolength{\textwidth}{-1in}
        \addtolength{\textwidth}{-#1}
}
\setleftmargin{0.7 in}
\setrightmargin{0.7 in}

\def\ZZ{\mathbb{Z}}
\def\QQ{\mathbb{Q}}
\def\RR{\mathbb{R}}
\def\FF{\mathbb{F}}
\def\CC{\mathbb{C}}
\def\PP{\mathbb{P}}

\def\XX{\mathbb X}

\def\Cater{\text{\boldmath $C$}}
\def\cL{\mathcal L}
\def\cT{\mathcal T}
\def\vert{\mathrm{vert}}
\def\even{\mathrm{even}}
\def\odd{\mathrm{odd}}
\def\lk{\ell k}
\def\Spec{\mathrm{Spec}}
\def\Flag{\mathcal{F}\hspace{-.03cm}\ell\hspace{-.03cm}\mathit{ag}}

\def\into{\hookrightarrow}

\def\Vol{\mathrm{Vol}}
\def\Span{\mathrm{Span}}
\def\Id{\mathrm{Id}}

\def\GL{\mathrm{GL}}
\def\SL{\mathrm{SL}}
\def\SO{\mathrm{SO}}
\def\Spin{\mathrm{Spin}}
\def\Pin{\mathrm{Pin}}
\def\IG{\mathrm{IG}}
\def\LG{\mathrm{LG}}
\def\G{\mathrm{G}}

\def\Cl{\mathit{Cliff}}

\def\Yo{\put(3,9){$\scriptscriptstyle \circ$}{Y}}
\def\Xo{\put(4,9){$\scriptscriptstyle \circ$}{X}}
\def\inn{\mathbin{\vrule width1.1ex height.4pt\vrule height1.2ex}\hspace{.1cm}}
\def\bwedge{\bigwedge \nolimits}
\def\ooplus{\, \oplus \, }

\def\Tmin{T_{\text{\rm min}}}

\def\dontshow#1{}
\newtheorem{conjecture}{Conjecture}[section]
\newtheorem{theorem}[conjecture]{Theorem}
\newtheorem{Proposition}[conjecture]{Proposition}
\newtheorem{lemma}[conjecture]{Lemma}
\newtheorem{corollary}[conjecture]{Corollary}

\newtheorem{definition}[conjecture]{Definition}

\newtheorem*{lemmagn}{Lemma~\ref{gn}}
\newtheorem*{lemmaLA}{Lemma~\ref{LA}}
\newtheorem*{corollarycK}{Proposition~\ref{cK}}
\newtheorem*{PropositionKey}{Proposition~\ref{Key}}
\newtheorem*{PropositionQus}{Proposition~\ref{Qus}}
\newtheorem*{theorem1}{Theorem 1}
\newtheorem*{theorem2}{Theorem 2}

\begin{document}

\title{The Multidimensional Cube Recurrence}
\date{}
\author{Andre Henriques and David E Speyer \footnote{The second author was supported by a Research Fellowship from the Clay Mathematics Institute}}

\maketitle

\abstract{We introduce a recurrence which we term the multidimensional cube recurrence, generalizing the octahedron recurrence studied by Propp, Fomin and Zelevinsky, Speyer, and Fock and Goncharov and the three-dimensional cube recurrence studied by Fomin and Zelevinsky, and Carroll and Speyer. The states of this recurrence are indexed by tilings of a polygon with rhombi, and the variables in the recurrence are indexed by vertices of these tilings. We travel from one state of the recurrence to another by performing elementary flips. We show that the values of the recurrence are independent of the order in which we perform the flips; this proof involves nontrivial combinatorial results about rhombus tilings which may be of independent interest. We then show that the multidimensional cube recurrence exhibits the Laurent phenomenon -- any variable is given by a Laurent polynomial in the other variables. We recognize a special case of the multidimensional cube recurrence as giving explicit equations for the isotropic Grassmannians $\IG(n-1,2n)$. Finally, we describe a tropical version of the multidimensional cube recurrence and show that, like the tropical octahedron recurrence, it propagates certain linear inequalities.}

%THESE ARE PRIVATE NOTES ON WORK BEGUN BY THE TWO AUTHORS, WHICH WE HOPE TO COMPLETE AND TURN INTO A LARGER PAPER. PLEASE DO NOT CIRCULATE WITHOUT PERMISSION OF ONE OF THE AUTHORS.

\section{Introduction}
\subsection{Statement of results}\label{S11}

Let $n\geq 3$ be an integer, and $A=(a_1,\dots,a_n)$ be a sequence of positive integers. 
Let $e_1, \dots ,e_n$ be the standard basis of $\ZZ^n$.  
Define $\Pi=\Pi(A)$ to be the subset \dontshow{dfP}
\begin{equation}\label{dfP}
\Pi:=\prod_{i=1}^n\{0, \ldots ,a_i\}
\end{equation}
of $\ZZ^n$, so that $| \Pi|=\prod_{i=1}^n (a_i+1)$.
Consider a collection of variables $x_I$ indexed by $I \in \Pi$ obeying the relations \dontshow{trc}
\begin{equation}\label{trc}
%\begin{split}
%x(I+e_j) x(I+e_i+e_k) &=  
%x(I) x(I+e_i+e_j+e_k) \\ 
%&+ x(I+e_i) x(I+e_j+e_k) \\ 
%&+ x(I+e_i+e_j) x(I+e_k) 
%\quad\text{for}\quad 1 \leq i<j<k \leq n.
x_{I+e_j+e_\ell}\, x_{I+e_k}
=  
x_{I}\, x_{I+e_j+e_k+e_\ell} 
+ x_{I+e_j+e_k}\, x_{I+e_\ell} 
+ x_{I+e_k+e_\ell}\, x_{I+e_j} 
\end{equation}
for $1 \leq j<k<\ell \leq n$.
Since the eight variables involved in this recurrence lie at the vertices of a cube, 
we refer to these relations as the {\em multidimensional cube recurrence}. 
The use of the term ``recurrence'' will become clear in Section~\ref{scz}. 
In the case where $n=3$, this was studied in unpublished work of Propp, and in~\cite{CarrSpey}.

Let $\Yo=\Yo(A)$ denote the set of solutions of these equations in $(\CC^\times)^{\prod_{i=1}^n (a_i+1)}=(\CC^\times)^{\Pi(A)}$. 
Call an element $(i_1, \ldots, i_n) \in \Pi$ \emph{even} or \emph{odd} depending on the parity of $i_1+\cdots + i_n$,
and let $(\CC^\times)^2$ act on $(\CC^\times)^{\Pi}$ by $(t,u)$ multiplying the even coordinates by $t$, and the odd coordinates by $u$. 
This action preserves $\Yo$. 
Let $\Xo$ be the quotient $\Yo/(\CC^\times)^2$, and let
$X$ be its closure in the product 
$\CC\PP^{\lceil (1/2) \prod_{i=1}^n (a_i+1) \rceil} \times \CC\PP^{\lfloor  (1/2) \prod_{i=1}^n (a_i+1)  \rfloor}
=\PP(\CC^{\Pi^\even})\times \PP(\CC^{\Pi^\odd})
$.
We will write $\Xo(A)$ and $X(A)$ when we need to emphasize the dependance on $A$.
Our main results are the following:

\begin{theorem1}
$\Yo$ is an irreducible variety of dimension $\sum_{i<j} a_i a_j+\sum_i a_i +1$. 
There is a certain collection of transcendence bases for the coordinate ring of $\Yo$ indexed by tilings of a certain two dimensional zonotope; 
any one of the $x_I$'s is given by a Laurent polynomial in terms of any of these bases.
\end{theorem1}

\begin{theorem2}
If all of the $a_i$'s are $1$, then $X$ is isomorphic to $\IG(n-1,2n)$, the space of $(n-1)$-planes in $\CC^{2n}$ 
that are isotropic with respect to a given non-degenerate quadratic form.
\end{theorem2}

We would like to recognize the variety $X$ when the $a_i$'s are larger than $1$.

\subsection{Motivation}

Let us explain why we began this investigation, and where we hope that it will go. 
Our work on the multidimensional cube recurrence is
motivated by the analogy with the multidimensional octahedron recurrence.
Let $n$, $m$ be positive integers, and let
$\Delta =\Delta(n,m)$ be the set of tuples $(i_0, i_1, \dots, i_n)\in (\ZZ_{\ge 0})^{n+1}$ satisfying $\sum i_k=m$. 
The term ``multidimensional octahedron recurrence" refers to the equations \dontshow{OctRecur}
\begin{equation}
\qquad x_{I+e_i+e_k}\, x_{I+e_j+e_{\ell}}=x_{I+e_i+e_j}\, x_{I+e_k+e_{\ell}} + x_{I+e_j+e_k}\, x_{I+e_\ell+e_i}\,,\qquad i<j<k<\ell, \label{OctRecur}
\end{equation}
in $(\CC^\times)^\Delta$.
The zero locus of these equations was identified by 
Fock and Goncharov~\cite[Section 9]{FG} with an open part of $\Flag_m^{n+1}/SL_m$,
where $\Flag_m$ denotes the space of flags $0\subset V_1\subset \ldots\subset V_{m-1}\subset V_m=\CC^m$
equipped with volume forms $\omega_i\in\bigwedge^iV_i$.
The coordinate ring of $\Flag_m^{n+1}/SL_m$ is an example of a cluster algebra:
it comes with distinguished collections of transcendence bases $\{x_I\}_{I\in S}$, parametrized by special subsets $S\subset \Delta$,
and every $x_I$ can be expressed as a Laurent polynomial in terms of any of these bases. 
Those subsets $S$ are in bijective correspondence with certain bi-colored polyhedral subdivisions of an $(n+1)$-gon, 
and one can use the recurrence (\ref{OctRecur}) to go from any one of these bases to any other one.
\begin{gather}
\psfrag{squi}{$\rightsquigarrow$}
\psfrag{a}{$\scriptstyle a$}\psfrag{b}{$\scriptstyle b$}\psfrag{c}{$\scriptstyle c$}\psfrag{d}{$\scriptstyle d$}
\psfrag{e}{$\scriptstyle e$}\psfrag{f}{$\scriptstyle f$}\psfrag{g}{$\scriptstyle g$}\psfrag{h}{$\scriptstyle h$}
\psfrag{i}{$\scriptstyle i$}\psfrag{j}{$\scriptstyle j$}\psfrag{k}{$\scriptstyle k$}\psfrag{l}{$\scriptstyle l$}
\psfrag{m}{$\scriptstyle m$}\psfrag{n}{$\scriptstyle n$}\psfrag{o}{$\scriptstyle o$}\psfrag{p}{$\scriptstyle p$}
\psfrag{q}{$\hspace{.05mm}\scriptstyle r$}\psfrag{r}{$\scriptstyle q$}\psfrag{s}{$\scriptstyle s$}\psfrag{t}{$\scriptstyle t$}\psfrag{u}{$\scriptstyle u$}
\psfrag{fq}{$\scriptstyle \frac{ek+bn}{f}$}
\nonumber\psfig{file=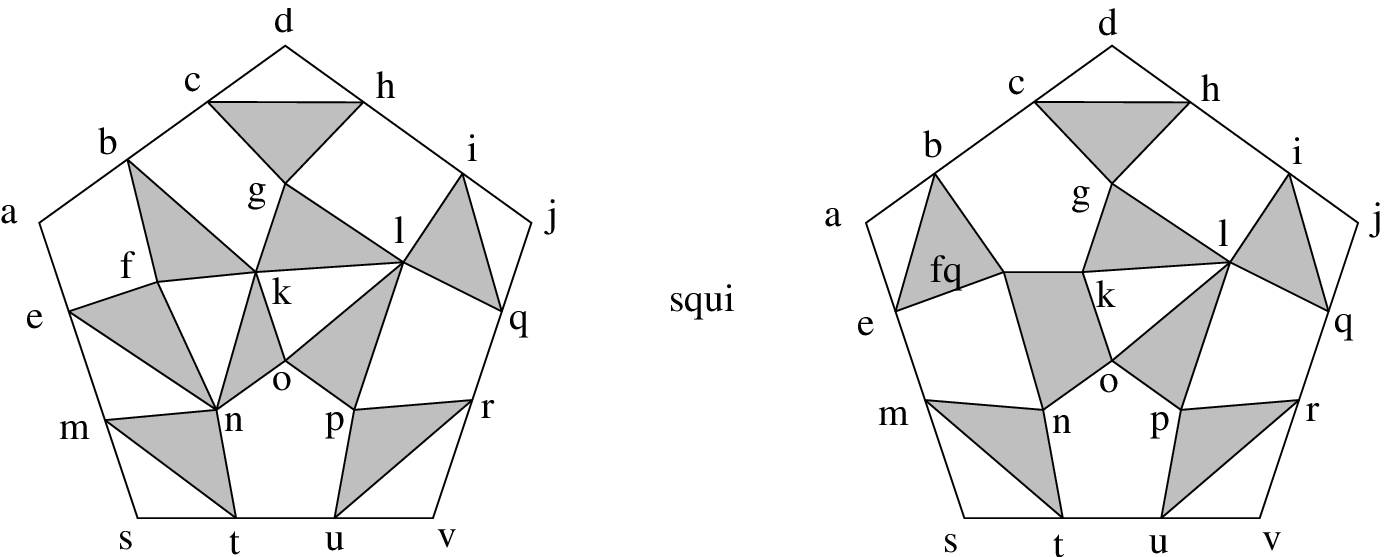, width=9cm}\\
\nonumber\text{The step of the octahedron recurrence (case $n=4$, $m=3$).}
\end{gather}
It is a major problem in the theory of cluster algebras to give combinatorial formulas for the above mentioned Laurent polynomials. 
In the case of the octahedron recurrence with $n=3$, this problem was solved by the second author in~\cite{Speyer}. 

%Then $\Yo(\Delta)$ is an irreducible variety with a collection of transcendence bases indexed by bicolored polyhedral subdivisions of a certain polygon. 
%Every $x_I$ can be expressed as a Laurent polynomial in terms of any of these bases. 
%Special choices of $\Delta$ can make $\Yo$ into (a torus bundle over an open subvariety of) the Grassmannian or 
%the moduli space of $n$-tuples of flags in $\CC^k$ modulo the action of $\GL_k$. 
%This latter example occurs in the work of Fock and Goncharov in their studies of ``higher Teichmuller theory" 
%as the transition formulas between the coordinates $\Delta_{abc}^v$ under a change of triangulation (see~\cite[Section 9]{FG}). 
%In general, $\Yo(\Delta)$ can be described as a torus bundle over an open subvariety of a moduli space of $n$-tuples of partial flags. 
%%(Presumably, we should take some torus quotient of $\Yo(\Delta)$ and close it up in some toric variety to obtain the actual moduli spaces, 
%but this does not appear to have been worked out.) 
%The coordinate rings of the $\Yo(A)$ are examples of cluster algebras 
%and the Laurentness properties discussed above 
%are special cases of general properties of cluster algebras. 
%It is a major combinatorial problem in the theory of cluster algebras to give combinatorial formulas for these Laurent polynomials. 
%In the case of the octahedron recurrence with $n=4$, this problem was solved by the second author in~\cite{Speyer}. 

In addition to the appearance of elegant algebraic varieties, 
the multidimensional octahedron recurrence has a connection to the representation theory of $\GL_m$. 
To see this, consider the tropical version of (\ref{OctRecur}) \dontshow{trooc}
\begin{equation}\label{trooc}
x_{I+e_i+e_k}+ x_{I+e_j+e_{\ell}}\,=\,\max\big(x_{I+e_i+e_j}+ x_{I+e_k+e_{\ell}}\,,\, x_{I+e_j+e_k}+ x_{I+e_\ell+e_i}\big),
\end{equation}
and introduce the following inequalities: \dontshow{ineco}
\begin{equation}\label{ineco}
x_I+x_{I+e_i-e_k}\,\ge\, x_{I+e_i-e_j}+x_{I+e_j-e_k}.
\end{equation}
A {\em hive} is a solution to (\ref{ineco}) in $\ZZ^{\Delta(2,m)}$, see~\cite{KTW}.
In other words, it is a triangular array of integers subject to the above inequalities.
The recurrence (\ref{trooc}) turns out to propagate these inequalities. 
This fact was then used by Knutson, Tao, and Woodward~\cite{KTW} in the case $n=3$, in order 
to identify the Littlewood-Richardson coefficients with the number of hives subject to certain boundary conditions. 
Their computation was later refined by the first author and by Kamnitzer~\cite{HK} in order to describe the associator in the category of $\mathfrak{gl}_m$-crystals. 
The case $n=4$ of the recurrence is related to the fact that the associator in this category satisfies the pentagon axiom.

Our motivation for considering the cube recurrence is that it appears to have combinatorial structures which are 
closely analogous to those of the octahedron recurrence.
Even though it does not fit into the formalism of cluster algebras,
one has special collections of variables 
such that all the other variables can be expressed as Laurent polynomials in terms of these collections. 
We also encounter certain classical varieties from the theory of Lie groups and, 
in the tropical version, there are inequalities which the recurrence inexplicably propagates.

The case $n=3$ of the cube recurrence was first investigated by Jim Propp, who conjectured the Laurentness property, 
and first mentioned in print by Fomin and Zelevinsky, who proved the Laurentness property~\cite{FZ}. 
G. Carroll and the second author investigated the combinatorics of the $n=3$ case in~\cite{CarrSpey}. 
To our knowledge, the higher dimensional case has not been discussed before this paper.

This research benefited from conversations with many other mathematicians. We would particularly like to thank Joel Kamnitzer, Allen Knutson, Jim Propp and Dylan Thurston.

\section{The recurrence\dontshow{scz}}\label{scz}

Let 
\begin{equation}\label{dfC}
C=C(A) :=\prod_{i=1}^n[0,a_i]
\end{equation}
be the obvious cubical complex with vertex set $\Pi$. 
Pick $0 < \theta_1 < \cdots < \theta_n < \pi$, and let $v_i := (\cos \theta_i, \sin \theta_i) \in \RR^2$. 
Let $\pi : C \to \RR^2$ be the map $(x_1, \ldots, x_n) \mapsto \sum x_i v_i$ and let $P:=\pi(C)$.
The polygon $P$ has $2n$ vertices, namely, $\pi(0)$, $\pi(e_1)$, $\pi(e_1+e_2)$, \dots, $\pi(e_1+e_2+\ldots+e_{n-1})$, $\pi(e_1+e_2+\ldots+e_{n-1}+e_n)$, $\pi(e_2+\ldots+e_{n-1}+e_n)$, \dots, $\pi(e_{n-1}+e_n)$, $\pi(e_n$).
The $i^{\textrm{th}}$ and $(n+i)^{\textrm{th}}$ edges are parallel and of the same length, namely $a_i$. 
A polygon whose edges have this property is called a \emph{zonogon}; see~\cite[Chapter 7]{Zieg} for background.

We define a \emph{tiling} to be a two dimensional sub-complex $T\subset C$ such that $\pi: T \to P$ is a homeomorphism.
These are the objects on which the initial conditions of our multidimensional cube recurrence can live.
A tiling is completely characterized by its 2-dimensional projection, which justifies our choice of terminology:

\begin{lemma} \label{lDm}\dontshow{lDm}
The map $\pi:C\to P$ induces a bijection between tilings $T\subset C$ and 
decompositions $\cT$ of $P$ into rhombi with side length $1$. 
\end{lemma}

\begin{proof}
We describe how to reconstruct $T$ from $\cT$.
Given a decomposition $\cT$ of $P$, we first note that any edge $\alpha\in\cT$ must be parallel to one of the vectors $v_i$.
Indeed, given $\alpha\in\cT$, pick a non-zero linear functional $\xi:\RR^2\to \RR$ that is constant on $\alpha$. 
Among all the edges of $\cT$ that are parallel to $\alpha$, let $\alpha'$ be one that maximizes $\xi$. 
The edge $\alpha'$ must be in $\partial P$, hence parallel to some $v_i$.
It follows that $\alpha$ is parallel to $v_i$. Given an oriented edge $\alpha\in\cT$, let
\[
\ell(\alpha):=\begin{cases}
\,\,e_i&\text{if $\alpha$ is parallel to $\,v_i$}\\
-e_i&\text{if $\alpha$ is parallel to $-v_i$}.
\end{cases}
\]
If $I$ is a vertex of $\cT$, let $\gamma=(\alpha_1,\ldots,\alpha_s)$ be a path in $\cT$ from $(0,0)$ to $I$, and let
$\tilde I:=\sum \ell(\alpha_i)\in C$. The vertex $\tilde I$ is then a preimage of $I$ under $\pi$.

Moreover, $\tilde I$ is independent of $\gamma$. 
Indeed, if $\gamma'=(\alpha'_1,\ldots,\alpha'_r)$ is another path from $(0,0)$ to $I$, 
we can write the cycle $\gamma-\gamma'$ as the boundary of a 2-chain $c=\sum n_i R_i$.
We then have $\sum\ell(\alpha_i)=\sum\ell(\alpha'_i)$ since
\[
\textstyle \sum\ell(\alpha_i)-\sum\ell(\alpha'_i)=\ell\big(\gamma-\gamma'\big)=\ell\big(\partial(\sum n_i R_i)\big)=\sum n_i\,\ell(\partial R_i)=0,
\]
where we have extended $\ell$ by linearity.

The collection of all $\tilde I$'s form the vertices of $T$. 
For every edge $(I_1, I_2)$ of $\cT$, there is a unique edge of $C$ connecting $\tilde{I}_1$ and $\tilde{I}_2$. 
Similarly, for each rhombus $(I_1,I_2,I_3,I_4)$ of $\cT$, there is a unique 2-face of $C$ containing $(\tilde{I}_1, \tilde{I}_2, \tilde{I}_3, \tilde{I}_4)$. 
These edges and 2-faces form the desired sub-complex $T\subset C$.
It is easy to check that $T$ is the unique sub-complex of $C$ which projects homeomorphically onto $\cT$.
\end{proof}

Lemma~\ref{lDm} is essentially a special case of the equivalence between 
``weak zonotopal tilings'' and ``strong zonotopal tilings'' proven in section 3 of~\cite{RichZieg}.
In view of the above lemma, we will sometimes identify tilings with their projection under $\pi$.
Let $I \in \Pi$, let $j < k <\ell$ be three numbers between $1$ and $n$, and let us assume that the cube \dontshow{xc}
\begin{equation}\label{xc}
c=\big\{I+xe_j+ye_k+ze_\ell\,\big|\,x,y,z\in[0,1]\big\}
\end{equation}
is contained in $C$.
The three facets of $c$ containing $I+e_k$ are then called the \emph{bottom faces} of $c$ and 
those containing $I+e_j+e_\ell$ the \emph{top faces}. 
If a tiling $T$ contains the top faces of $c$ then the complex formed by deleting the top faces of 
$c$ and replacing them with the bottom faces is also a tiling (the same is true with the 
words ``top'' and ``bottom'' reversed).
We will say that two tilings $T$, $T'$ that differ only by the above modification are related by a \emph{flip}, and we shall write it $T\rightsquigarrow T'$.
As an example, we illustrate the eight tilings of an octagon and all their possible flips:
\begin{equation}\label{figoc}
\psfrag{h}{$\scriptstyle \leftrightsquigarrow$}
\psfrag{v}{\rotatebox{90}{$\scriptstyle \leftrightsquigarrow$}}
\psfrag{ne}{\rotatebox{10}{$\scriptstyle \leftrightsquigarrow$}}
\psfrag{se}{\rotatebox{-10}{$\scriptstyle \leftrightsquigarrow$}}
\begin{matrix}\epsfig{file=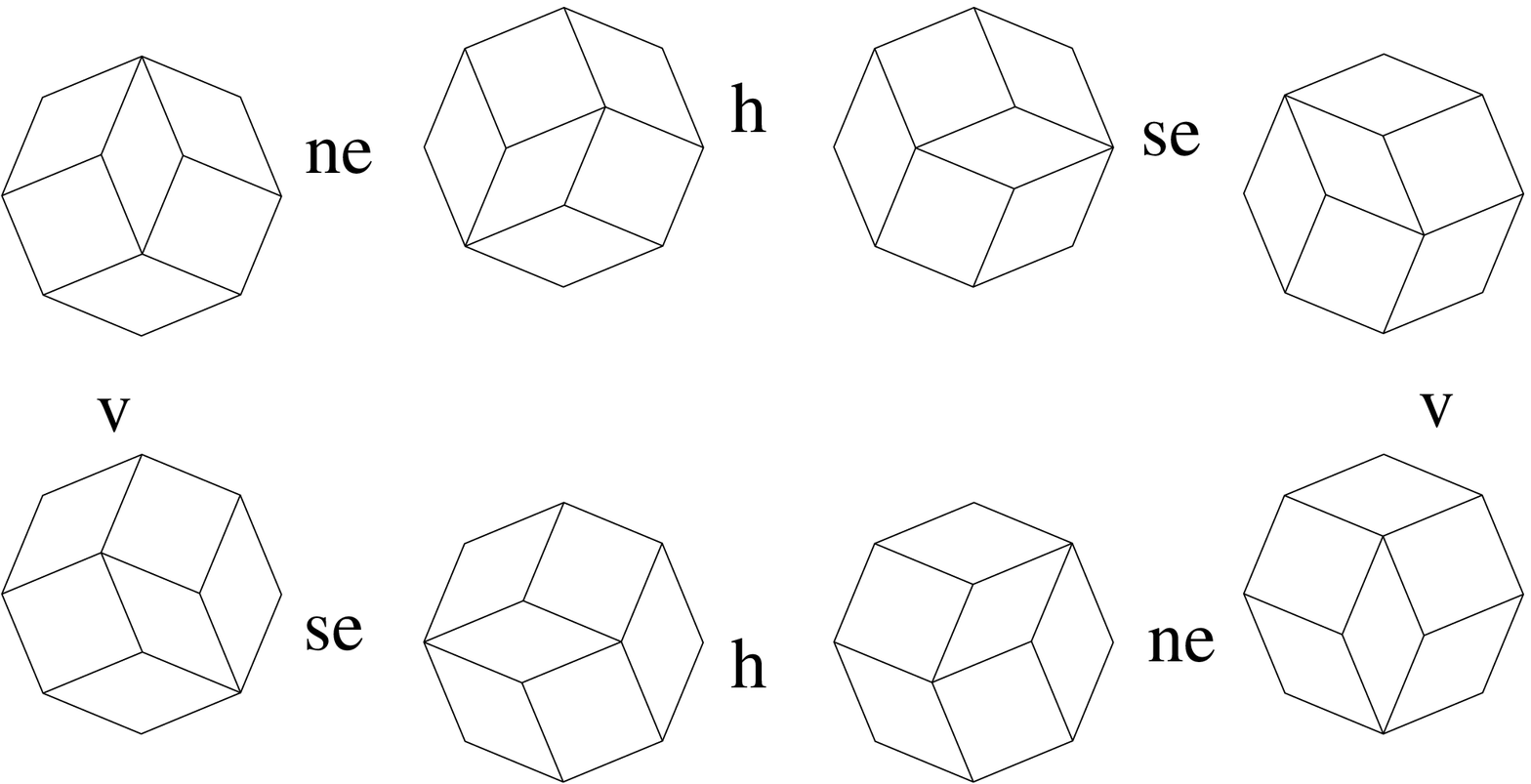, width=8cm}\end{matrix}
\end{equation}
%\begin{gather}\label{figoc}
%\begin{matrix}\epsfig{file=octagon.eps, width=8cm}\end{matrix}\\
%\nonumber\text{The eight tilings of an octagon and their flips}
%\end{gather}

If $T$ and $T'$ are related by a flip, the set of vertices of $T'$ are $(\vert(T)\setminus \{ J \})\cup \{ J' \}$, where $\{J,J'\}=\{I+e_k,I+e_j+e_\ell\}$.
So given a collection of variables $\{x_I \}_{I\in\vert(T)}$, 
we may solve equation (\ref{trc}) for $x(J')$ and thus find associated values for all the vertices of $T'$.
This is the elementary step of our recurrence.
\begin{gather*}
\psfrag{a}{$\scriptstyle a$}\psfrag{b}{$\scriptstyle b$}\psfrag{c}{$\scriptstyle c$}\psfrag{d}{$\scriptstyle d$}
\psfrag{e}{$\scriptstyle e$}\psfrag{f}{$\scriptstyle f$}\psfrag{g}{$\scriptstyle g$}\psfrag{h}{$\scriptstyle h$}
\psfrag{i}{$\scriptstyle i$}\psfrag{j}{$\scriptstyle j$}\psfrag{k}{$\scriptstyle k$}\psfrag{l}{$\scriptstyle l$}
\psfrag{m}{$\scriptstyle m$}\psfrag{n}{$\scriptstyle n$}\psfrag{o}{$\scriptstyle o$}\psfrag{p}{$\scriptstyle p$}
\psfrag{q}{$\scriptstyle q$}\psfrag{r}{$\scriptstyle r$}\psfrag{s}{$\scriptstyle s$}\psfrag{squig}{$\rightsquigarrow$}
\psfrag{gp}{$\scriptstyle \frac{fh+bm+cl}{g}$}
\begin{matrix}\epsfig{file=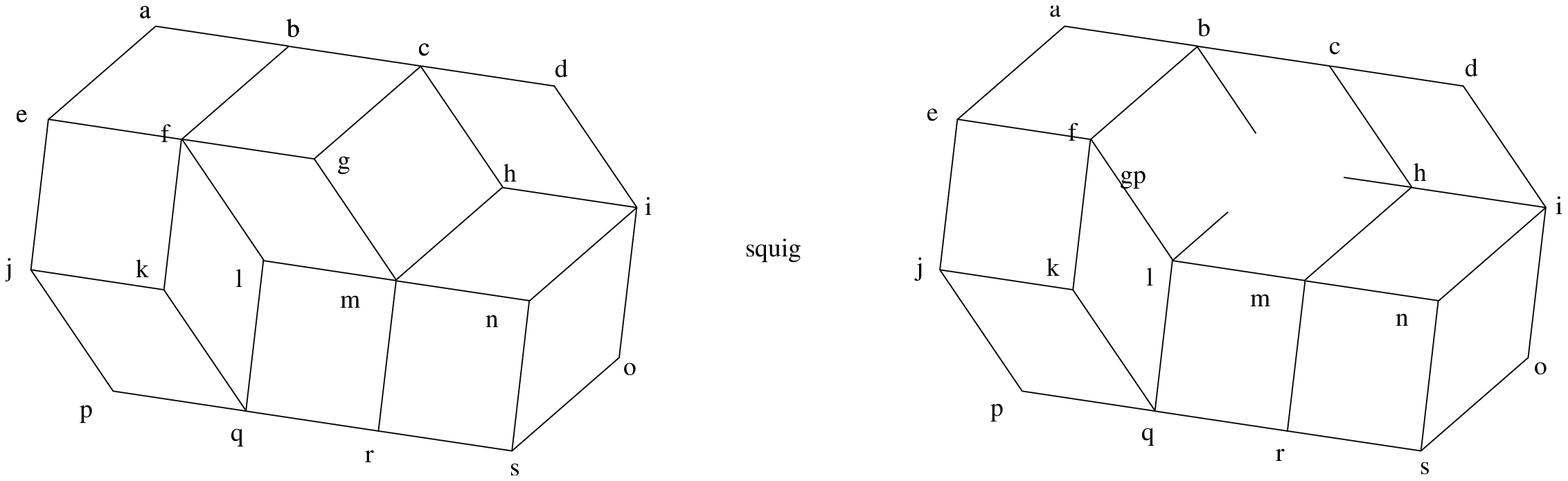, height=2.5cm}\end{matrix}\\
\nonumber\text{The step of the cube recurrence \!\big(\hspace{-.03cm}case $n=4$, $A=(1,1,1,3)$\big).}
\end{gather*}
In order to avoid divisions by zero, we shall assume that the $x_I $ are always {\em positive elements} of some ordered field $\FF$.
For example, $\FF$ might be $\QQ(\{x_I \}_{I\in \vert(T)})$, the field of rational functions on the variables $\{x_I\}_{I\in\vert(T)}$ for some given tiling $T$,
with an order on $\FF$ coming from an embedding $\FF\hookrightarrow\RR$ sending $x_I $ to positive real numbers.
If \dontshow{tqz}
\begin{equation}\label{tqz}
T=T_0\rightsquigarrow T_1\rightsquigarrow \ldots \rightsquigarrow T_k=T'
\end{equation}
are tilings related by a sequence of flips, and $x_I $ are associated to the vertices of $T$, we can then use (\ref{trc}) recursively to get values on the vertices of $T'$.
One of our main theorems asserts that this operation is well defined, \emph{i.e.}, it does not depend on the choice of intermediate tilings $T_1,\ldots, T_{k-1}$.

\begin{theorem}\label{tWt}\dontshow{tWt}
Let $T$, $T'$ be two tilings. Then there exists a sequence of intermediate tilings $T_i$ as in (\ref{tqz}), such that each two consecutive ones are related by a flip.

If $x_I \in\FF$ are positive elements attached to the vertices of $T$, then the values at the vertices of $T'$ obtained by the successive application of (\ref{trc}) 
do not depend on the choice of intermediate tilings $T_i$.
\end{theorem}

The proof of Theorem~\ref{tWt} requires a number of results on the combinatorics of tilings. 
These results, which we list below, will be proven in Section~\ref{fcj}. 
%
%\begin{lemmagn}
%Let $c$ be a three dimensional face of the cubical complex $C$, then there exists a tiling $T\subset C$ that contains the three bottom faces of $c$.
%Similarly, there exists a tiling containing the three top faces of $c$.
%\end{lemmagn}
%
%
%

Let $\XX_1$ be the graph whose vertices are the tilings of $P$, and whose edges are the flips.

\begin{corollarycK}
The graph $\XX_1$ is connected. 
\end{corollarycK}

We now construct a two dimensional cellular complex $\XX_2$ 
by gluing the following two-cells onto $\XX_1$:
 
{\it i)}
If $T\rightsquigarrow T'$ and $T\rightsquigarrow T''$ are flips involving disjoint sets of rhombi, 
we can perform the two of them simultaneously to get a fourth tiling $T'''$.
The vertices $T,T',T'',T'''$ form a 4-cycle in $\XX_1$ on which we attach a square.

{\it ii)} 
Suppose that $\sigma$ is a four-cell of $C$, and that $T$ is a tiling which contains a tiling of $\pi(\sigma)$.
Then $T$ contains one of the figures (\ref{figoc}) as a subset, and we may perform the corresponding cycle of eight flips. 
In each such case, we glue an octagon with boundary this series of eight flips.

\begin{PropositionKey}
The cell complex $\XX_2$ is simply connected.
\end{PropositionKey}

%\proof[Proof of Theorem~\ref{tWt}]
\begin{proof}[Proof of Theorem~\ref{tWt}]
The fact that $T$ and $T'$ can be joined by a sequence of flips is the content of Proposition~\ref{cK}.
Given two paths $\gamma=\{T_i\}$, $\gamma'=\{T'_i\}$ in $\XX_1$ between $T$ and $T'$, 
we want to show that the values at $\vert(T')$ computed using $\gamma$ agree with the values computed using $\gamma'$. 
Since the recursion~(\ref{trc}) is invertible, it is
equivalent to show that the values at $\vert(T)$ computed by following the loop $\gamma\gamma'^{-1}$ agree with the original values.
In other words, we want to show that $\pi_1(\XX_1,T)$ acts trivially on the set of possible values of $\vert(T)$.
By Proposition~\ref{Key}, any element of $\pi_1(\XX_1,T)$ is a product of loops of the forms \dontshow{2lt}
\begin{equation*}%\label{2lt}
\psfrag{T}{$T$}
\begin{matrix}
\epsfig{file=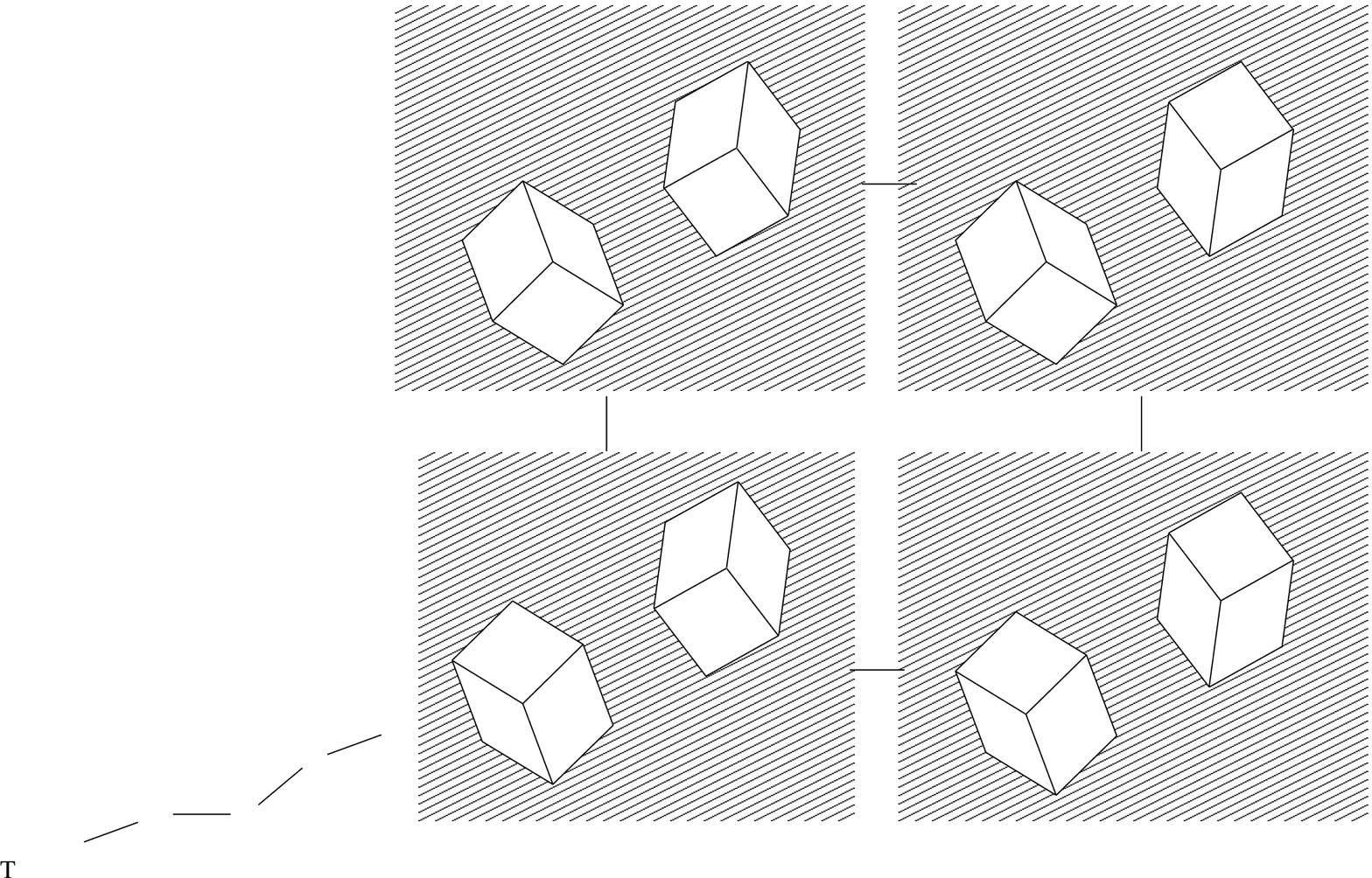, height=3cm}
\end{matrix}
\qquad\text{and}
\begin{matrix}
\epsfig{file=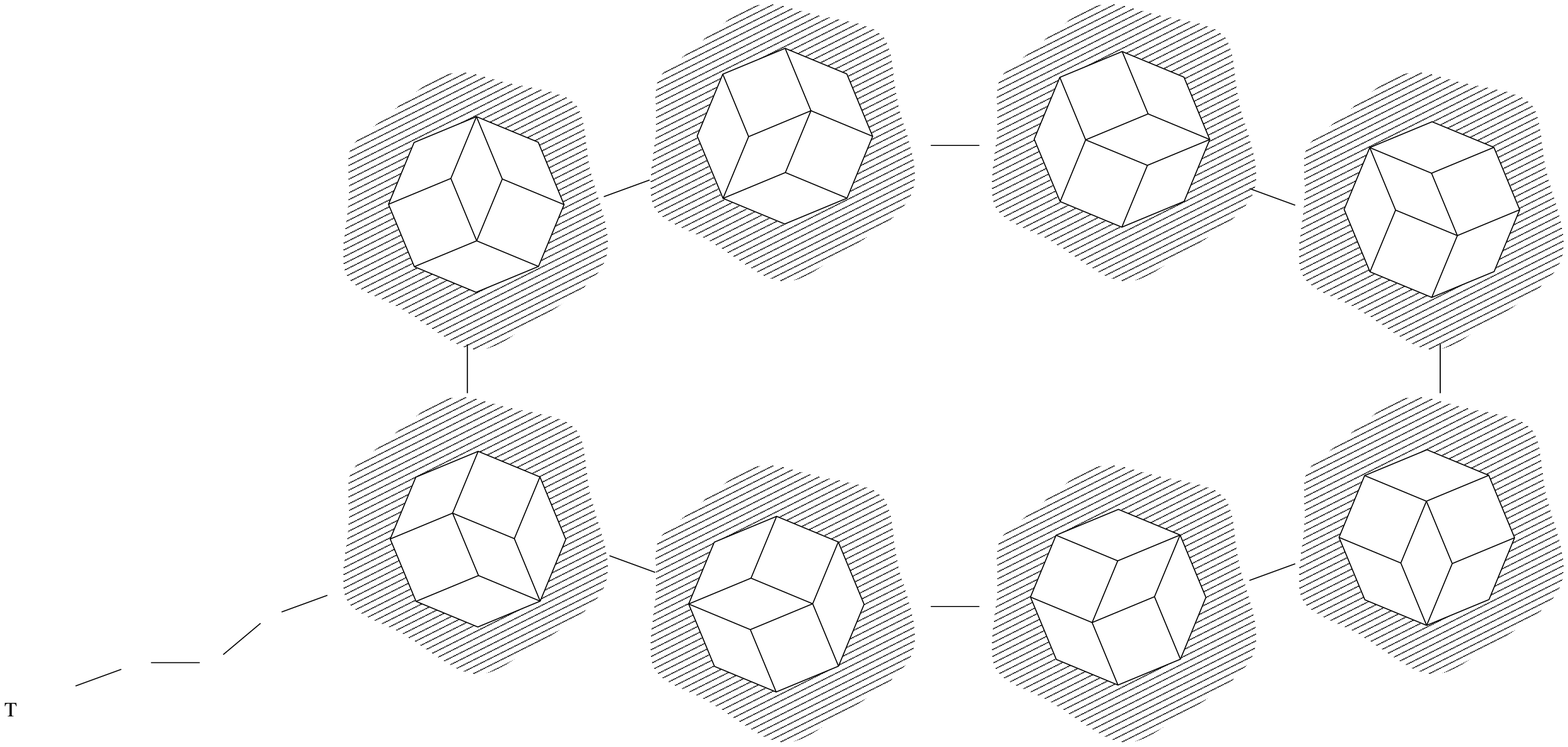, height=3.5cm}
\end{matrix}
\end{equation*}
The first loop clearly acts trivially.
The second one acts trivially by the following elementary computation:
{
\psfrag{a}{$\scriptstyle a$}\psfrag{b}{$\scriptstyle b$}\psfrag{c}{$\scriptstyle c$}\psfrag{d}{$\scriptstyle d$}
\psfrag{e}{$\scriptstyle e$}\psfrag{f}{$\scriptstyle f$}\psfrag{g}{$\scriptstyle g$}\psfrag{h}{$\scriptstyle h$}
\psfrag{i}{$\scriptstyle i$}\psfrag{j}{$\scriptstyle j$}\psfrag{k}{$\scriptstyle k$}\psfrag{l}{$\scriptstyle l$}
\psfrag{m}{$\scriptstyle m$}\psfrag{n}{$\scriptstyle n$}\psfrag{o}{$\scriptstyle o$}\psfrag{p}{$\scriptstyle p$}
\psfrag{q}{$\scriptstyle q$}\psfrag{r}{$\scriptstyle r$}\psfrag{s}{$\scriptstyle s$}
\begin{flushleft} $\begin{aligned}\epsfig{file=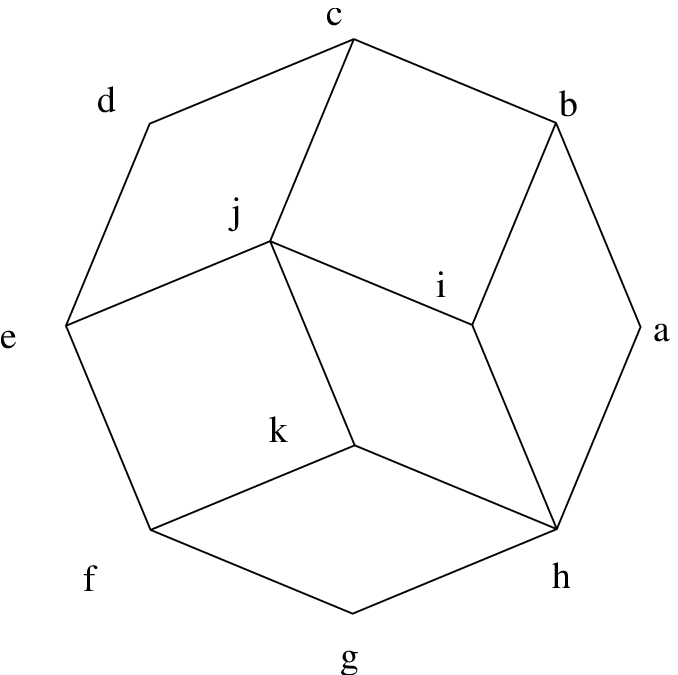, height=1.8cm}\end{aligned}\quad
 \rightsquigarrow\quad\begin{aligned}\epsfig{file=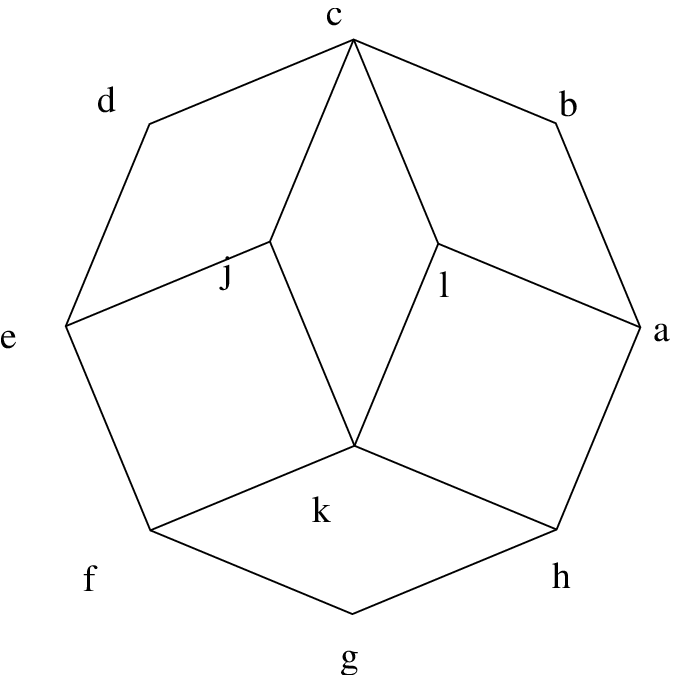, height=1.8cm}\end{aligned}\quad \textstyle l=\frac{aj+bk+ch}{i} $\end{flushleft}
\begin{flushleft} $\phantom{\begin{matrix}\epsfig{file=oc1.eps, height=1.8cm}\end{matrix}}\quad\rightsquigarrow\quad\begin{matrix}\epsfig{file=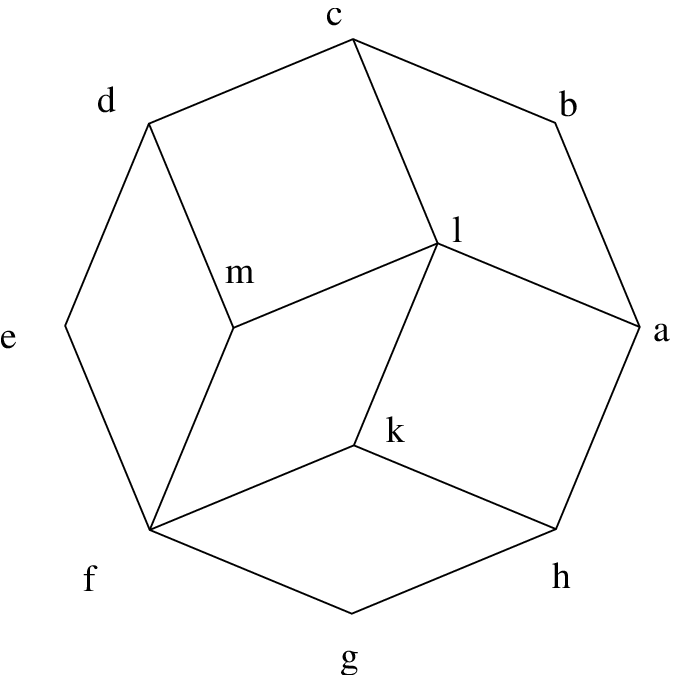, height=1.8cm}\end{matrix}\quad \textstyle m=\frac{cf+dk+el}{j}=\frac{cfi+dki+eaj+ebk+ech}{ij}$ \end{flushleft}
\begin{flushleft} $\phantom{\begin{matrix}\epsfig{file=oc1.eps, height=1.8cm}\end{matrix}}\quad\rightsquigarrow\quad\begin{matrix}\epsfig{file=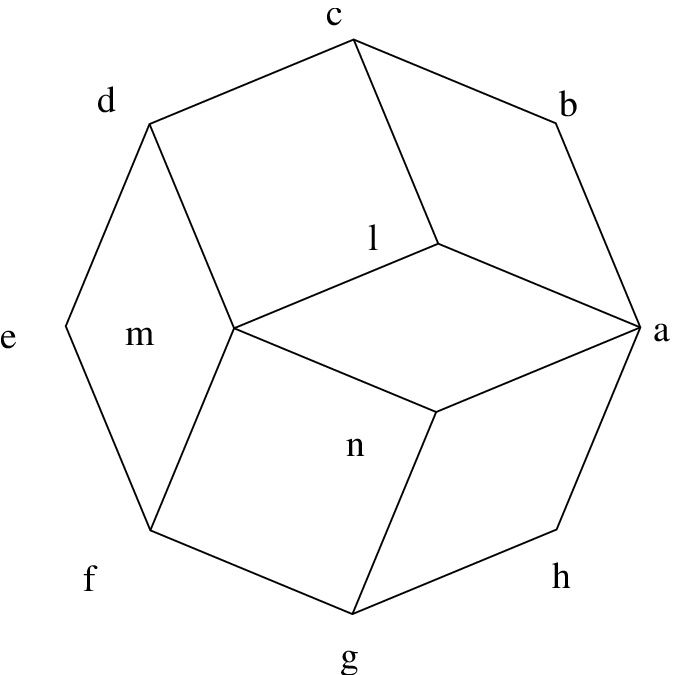, height=1.8cm}\end{matrix}\quad \textstyle n=\frac{fa+gl+hm}{k}=\frac{faij+gaj^2+gbkj+gchj+hcfi+hdki+heaj+hebk+ech^2}{ijk}$\end{flushleft}
\begin{flushleft} $\phantom{\begin{matrix}\epsfig{file=oc1.eps, height=1.8cm}\end{matrix}}\quad\rightsquigarrow\quad\begin{matrix}\epsfig{file=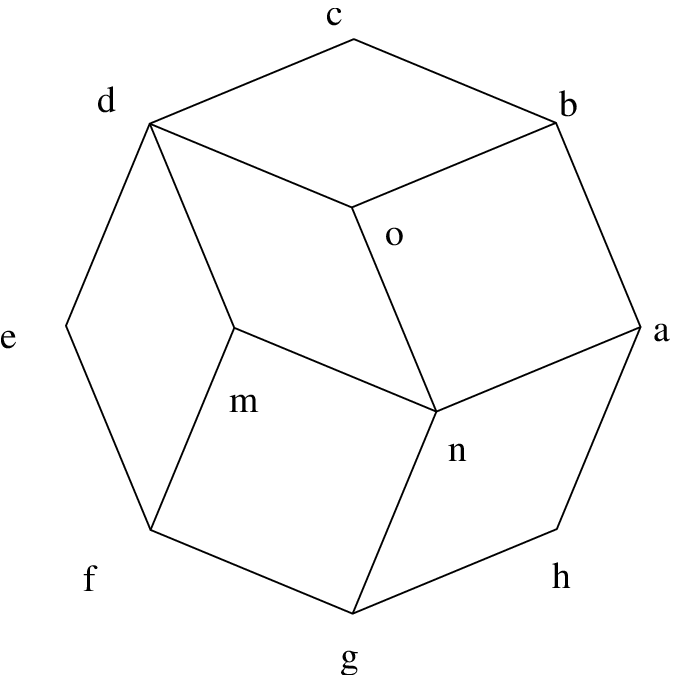, height=1.8cm}\end{matrix}\quad \textstyle o=\frac{ad+bm+cn}{l}=\frac{bek+dik+cgj+cfi+ceh}{jk}$\end{flushleft}
\begin{flushleft} $\phantom{\begin{matrix}\epsfig{file=oc1.eps, height=1.8cm}\end{matrix}}\quad\rightsquigarrow\quad\begin{matrix}\epsfig{file=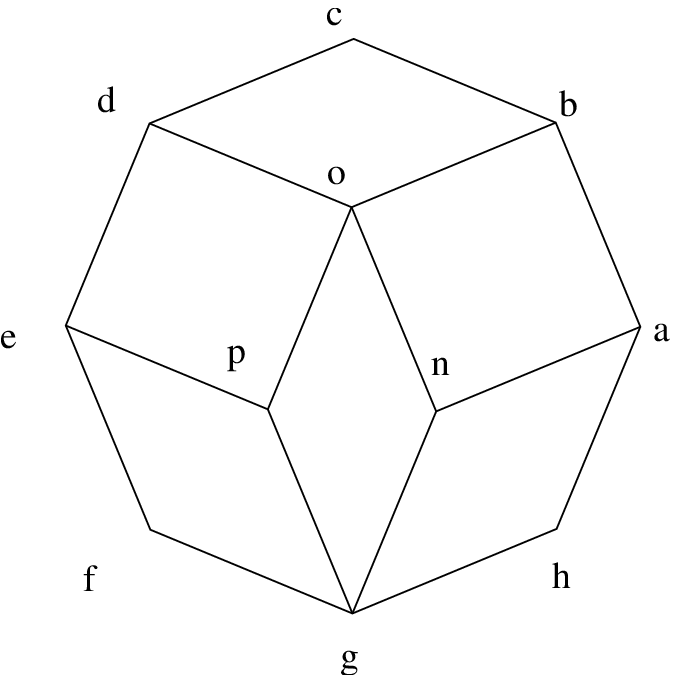, height=1.8cm}\end{matrix}\quad \textstyle p=\frac{dg+en+fo}{m}=\frac{eh+fi+gj}{k}\quad
\rightsquigarrow\quad\begin{matrix}\epsfig{file=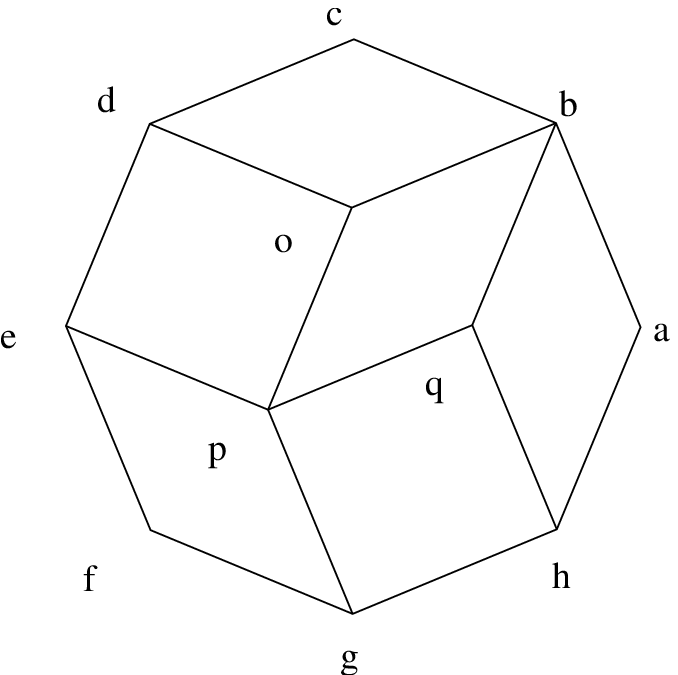, height=1.8cm}\end{matrix}\quad \textstyle q=\frac{ho+ap+bg}{n}=i$\end{flushleft}
\begin{flushleft} $\phantom{\begin{matrix}\epsfig{file=oc1.eps, height=1.8cm}\end{matrix}}\quad\rightsquigarrow\quad\begin{matrix}\epsfig{file=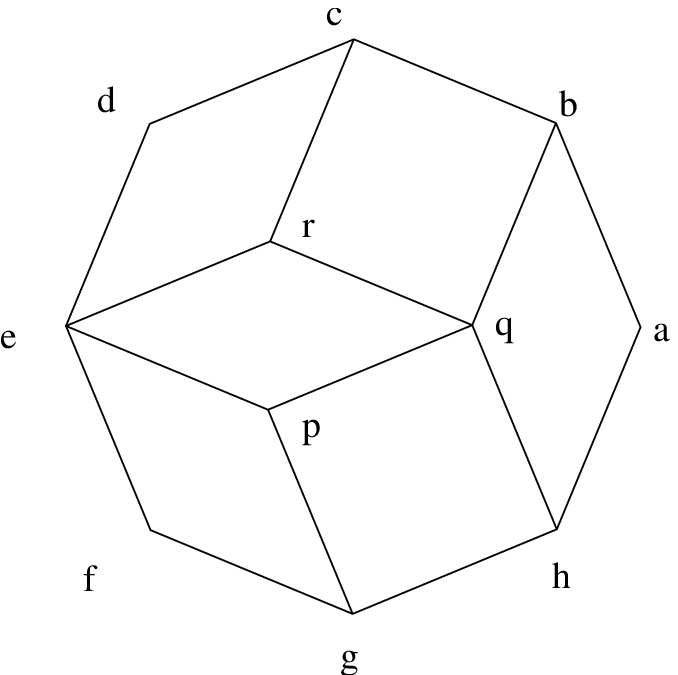, height=1.8cm}\end{matrix}\quad \textstyle r=\frac{be+cp+dq}{o}=j\quad
\rightsquigarrow\quad\begin{matrix}\epsfig{file=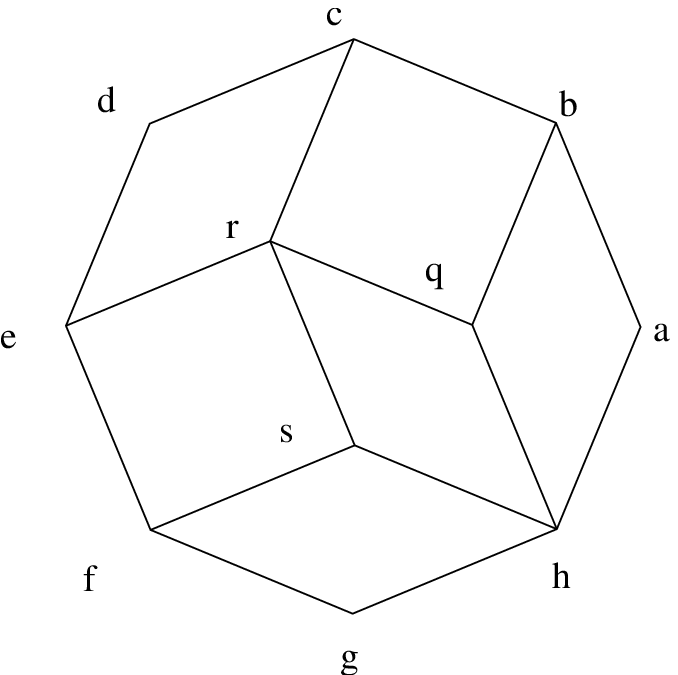, height=1.8cm}\end{matrix}\quad \textstyle s=\frac{eh+fq+gr}{p}=k$\hfill \raisebox{-1cm}{$\square$}\end{flushleft}
}
\end{proof}

Any element of $\Pi$ is contained in a tiling (see Lemma~\ref{LA} below). 
So we can use our recurrence to extend the values $\{x_I \}_{I\in\vert(T)}$ to all of $\Pi$.

\begin{theorem}\label{h:i}\dontshow{h:i}
Let $T$ be a tiling, and let $\{x_I \}_{I\in\vert(T)}$ be a collection of positive elements of $\FF$ attached to the vertices of $T$.
Then there is a unique way of extending $\{x_I \}$ to a labeling of the whole $\Pi$ by elements of $\FF$, such that all the equations (\ref{trc}) are satisfied.
\end{theorem}

Once again, we need a few results from Section~\ref{fcj} before we can prove the theorem:

\begin{lemmaLA}
For each vertex $I\in \Pi$, there exists a tiling $T$ containing $I$.
\end{lemmaLA}

\begin{lemmagn}
Let $c$ be a 3-face of $C$. 
Then there exists a tiling $T\subset C$ that contains the three bottom faces of $c$.
Similarly, there exists a tiling containing the three top faces of $c$.
\end{lemmagn}

\begin{PropositionQus}
Let $T$, $T'$ be two tilings of $P$, and let $I \in \Pi$ be a vertex contained in both $T$ and $T'$.
Then there is a sequence of flips going from $T$ to $T'$, so that all intermediate tilings contain $I$.
\end{PropositionQus}

\begin{proof}[Proof of Theorem~\ref{h:i}]
Let $I\in\Pi$ be a vertex.
By Lemma~\ref{LA}, there exists a tiling $T'$ containing $I$.
By Theorem~\ref{tWt}, we can then use our recurrence to compute the values at $\vert(T')$,
and these values are independent of the way one goes from $T$ to $T'$.
To see that $x_I $ is independent of $T'$, consider another tiling $T''$ containing $I$.
By Proposition~\ref{Qus}, there exists a sequence of flips \dontshow{mpt}
\begin{equation}\label{mpt}
T'=T_0\rightsquigarrow T_1 \rightsquigarrow \ldots \rightsquigarrow T_s=T''
\end{equation}
so that each $T_i$ contains $I$.
Computing $x_I $ by first going from $T$ to $T'$, and then following (\ref{mpt}), we see that the two values of $x_I $ must agree.

It remains to see that $\{x_I \}_{I\in\Pi}$ satisfies the relations (\ref{trc}). 
Indeed by Lemma~\ref{gn}, any such relation is involved in at least one flip $S\rightsquigarrow S'$.
Computing the values at $\vert(S')$ by first going from $T$ to $S$ and then doing that flip, we see that (\ref{trc}) is satisfied.
\end{proof}

%Below we have drawn the eight tilings of an octagon, corresponding 
%to the case $a_1=a_2=a_3=a_4=1$. Adjacent tilings are connected by a flip. \dontshow{figoc}

\section{Combinatorics of Zonogons\dontshow{fcj}}\label{fcj}

In this section, we will prove the combinatorial results about rhombus tilings which we used in the preceding section. 
Many of the results which are established in this section are not original to us, so we pause to summarize what was already known.

%NEW STUFF HERE

Rhombus tilings of zonogons have been studied in many contexts, and are known to be in bijection with many other objects.
We have already mentioned one of these bijections; the relation between rhombus tilings and sections of $\pi: C \to \RR^2$ (Lemma~\ref{lDm}).
As we described when stating Lemma~\ref{lDm}, this seems to have first been recorded by Richter-Gebert and Ziegler~\cite{RichZieg}.
Another bijection, which we do not use directly, is between rhombus tilings and certain oriented matroids.
This bijection is known as the Bohne-Dress theorem; it was announced by Dress, proved in the unpublished dissertation of Bohne and reproved in~\cite{RichZieg}.
By the standard bijection between oriented matroids and pseudoline arrangements (see~\cite{LF}), rhombus tilings are therefore in bijection with certain pseudo-line arrangements; we use this perspective occasionally (Lemmas~\ref{LA} and~\ref{gn})\footnote{When some of the $a_i$ are greater than $1$, Richter-Gebert and Ziegler take a slightly different approach than we do. They consider the same $2n$-gon we do, but they consider the projection from the cube $[0,1]^{\sum a_i}$ instead of $\prod [0,a_i]$. When discussing the relation between tilings and sections of $\pi$, this introduces purely cosmetic difficulties (Lemma~\ref{lDm} contains the correct statement). When discussing the situations of matroids and pseudo-line arrangements this becomes annoying; essentially, we want the $a_i$ parallel pseudo-lines to be numbered sequentially by the integers in $[1, a_i]$, while Richter-Gebert and Ziegler will permit us to number them arbitrarily.}.
Finally, tilings are in bijection with certain commutation classes of reduced words in the symmetric group $S_{\sum a_i}$; see~\cite{Elnit}.
These commutation classes, in the case $a_1=a_2=\ldots=a_n=1$, correspond to the elements of the higher Bruhat order $B(n,2)$; see~\cite{ZiegHB}. 

The main results of this section are Corollary~\ref{cht} (downward flips will eventually lead to $\Tmin$) with its corollary Proposition~\ref{cK} (the graph of flips is connected); 
Proposition~\ref{Key} (describing the fundamental group of the graph of flips); and Proposition~\ref{Qus} (a technical variant on Corollary~\ref{cht}). 
Proposition~\ref{cK} has been established many times. 
The first proof in the context of rhombus tilings may be due to Kenyon~\cite[Theorem 5]{Kenyon}; Ringel gave a proof in the context of pseudoline arrangements~\cite{Ringel}.
The more detailed Corollary~\ref{cht} has been proved in the context of higher Bruhat orders as the statement that $B(n,2)$ has a unique minimal element.
Proposition~\ref{Key} can be deduced from a result on oriented matroids, namely~\cite[Theorem~1.2]{SturmZieg}, and from a known result about reduced words in $S_n$, namely~\cite[Lemma 3.14]{SSV}.
In both cases, however, the translation to rhombus tilings is nontrivial, and would be of length comparable to our proof. 
As far as we can tell, Proposition~\ref{Qus} is completely new. 
Our use of the fundamental forest appears to be original, and thus our proofs of these results are new; we hope that the fundamental forest may be of use in the future study of rhombus tilings.

Given a cube $c$ as in (\ref{xc}) we have defined its three bottom faces, its three top faces, and the notion of flip between two tilings. 
Given a flip $T\rightsquigarrow T'$, there is a unique vertex $J\in T\setminus T'$, we then say that the flip is \emph{performed at} $J$.
Note that it is possible to perform a flip at $J$ if and only if $J$ is a trivalent vertex of $T$.
An {\em upward flip} will be the operation of replacing the bottom faces by the top faces, and a {\em downward flip} will be the opposite. \dontshow{figflip}

Edges out of a vertex $J$ come in two flavors: the ones of the form $(J,J+e_i)$, and the ones of the form $(J,J-e_i)$.
We say that the former are {\em pointing up} and that the latter are {\em pointing down}.
Note that this terminology is consistent with our way of projecting things onto $\RR^2$ 
since the $y$-coordinate of $\pi(J+e_i)$ (respectively $\pi(J-e_i)$) is always bigger (resp. smaller) than that of $\pi(J)$.
For example, the vertices at which upward flips can be performed have 2 edges pointing up and one pointing down.
Similarly, the vertices at which downward flips can be performed have 2 edges pointing down and one pointing up.

\begin{equation}\label{figflip}
\psfrag{up}{upward flip}
\psfrag{down}{downward flip}
\psfrag{t}{$\rightsquigarrow$}
\begin{matrix}\epsfig{file=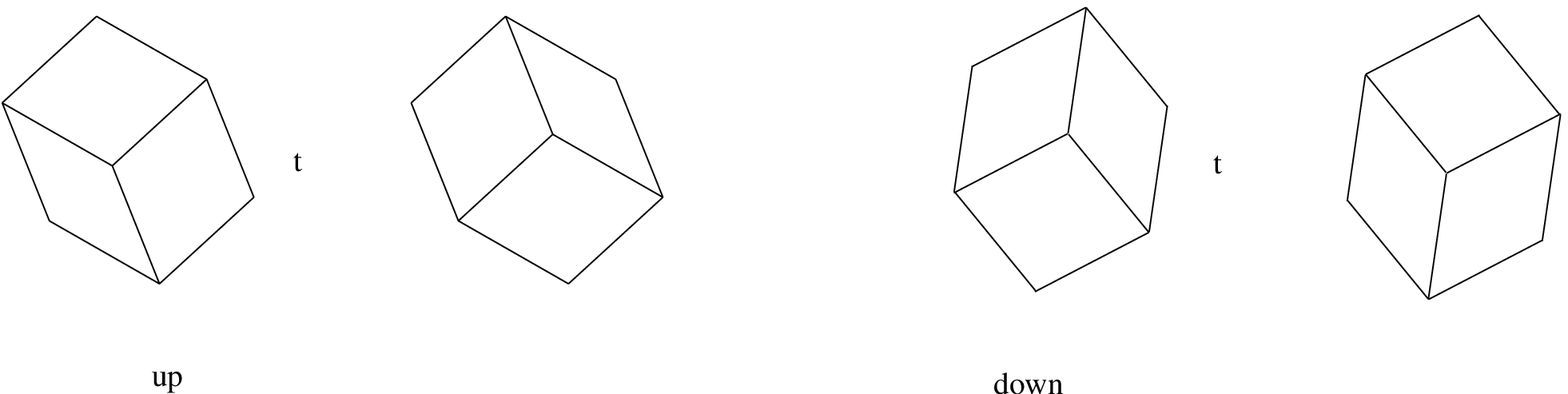, width=8cm}\end{matrix}
\end{equation}

We begin by showing that zonotopal tilings always exist in great numbers.

\begin{lemma}\label{LA}
For each vertex $I\in \Pi$, there exists a tiling $T$ containing $I$.
\end{lemma}

\begin{proof}
Write $(i_1,\ldots,i_n)$ for $I$. For each $r\in\{1,\ldots,n\}$, pick generic real numbers
$q_{r,1}<q_{r,2}<\ldots < q_{r,i_r}<0< q_{r,i_r+1}<\ldots< q_{r,a_r}$, where $a_r$ is as in Section~\ref{S11}, and define
\[
L_{r,m}:=\big\{x\in\RR^2\,\big|\,\langle x,v_r\rangle=q_{r,m}\big\}\,,\qquad \cL:=\bigcup L_{r,m}\,.
\]
Since the $q_{r,m}$ are generic, there are no triple intersections in $\cL$, 
and its planar dual is thus homeomorphic to a tiling $T$ of $P$ by rhombi (see~\cite{Zieg}).
A component of $\RR^2\setminus \cL$ corresponds to a vertex of $T$, which by Lemma~\ref{lDm} can be viewed as element of $\Pi$.
The $r^{\textrm{th}}$ coordinate of such a vertex is then given by the number of $L_{r,1},\ldots,L_{r,a_r}$ that pass ``below'' the component.

Now consider the component of $\RR^2\setminus \cL$ containing the origin.
By construction, the corresponding vertex of $T$ has coordinates $(i_1,\ldots,i_n)$.
Therefore $I\in T$.
\end{proof}

We next show that every three dimensional face of $C$ corresponds to an actual flip between some pair of tilings:

\begin{lemma}\label{gn} \dontshow{gn}
Let $c$ be a 3-face of $C$. 
Then there exists a tiling $T\subset C$ that contains the three bottom faces of $c$.
Similarly, there exists a tiling containing the three top faces of $c$.
\end{lemma}

\begin{proof}
Let $\{I+\alpha e_j-\beta e_k + \gamma e_\ell\}$, $\alpha,\beta,\gamma \in\{0,1\}$, $j < k < \ell$, be the vertices of $c$, and let us write $(i_1,\ldots,i_n)$ for $I$.
To construct $T$, we proceed as in Lemma~\ref{LA},
but we make sure to pick $q_{j,i_j+1}$, $q_{k,i_k}$, and $q_{\ell,i_\ell+1}$ very close to zero, compared with the other $q_{r,m}$'s.
The arrangement $\cL$ then has a small triangle $\tau$ containing the origin, made out of the lines $L_{j,i_j+1}$, $L_{k,i_k}$, and $L_{\ell,i_\ell+1}$.
In particular, $\tau$ has six neighboring regions.
The vertex corresponding to $\tau$ is $I$, and the ones corresponding to the six neighboring regions are 
$I-e_k$, $I+e_j-e_k$, $I+e_\ell-e_k$, $I+e_j$, $I+e_\ell$, and $I+e_j+e_\ell$.
These vertices compose the three bottom faces of $c$, so $T$ is our desired tiling.

The tiling containing the top faces of $c$ is constructed by a similar procedure,
or by performing a flip on $T$.
\end{proof}

\begin{lemma}\label{ncb}\dontshow{ncb}
One can never come back to the same tiling by only doing downward flips.
\end{lemma}

\begin{proof}
Let $\varphi(i_1,\ldots, i_n):=\sum i_j$, and $\phi(T):=\sum_{I\in T}\varphi(I)$. 
If $T\rightsquigarrow T'$ is a downward flip, then $\vert(T')=(\vert(T) \setminus \{I+e_j+e_\ell \}) \cup \{ I+e_k \}$ for some appropriate vertex $I$.
It follows that $\phi(T')=\phi(T)-1$.
Downward flips always decrease the value of $\phi$, so the result follows.
\end{proof}

It will be very important for us to know when exactly we can perform downward flips.
For this purpose, we introduce the following technical definition.

\begin{definition}
Let $T$ be a tiling.
The fundamental forest $\Gamma\subset T$ is the union of all internal edges of the form $(I,I+e_j)$, where $(I,I+e_j)$ is the only edge pointing up from $I$. 
(By an internal edge, we mean an edge $(I,I+e_j)$ such that $\pi( (I,I+e_j) ) \not \subset \partial P$.) \end{definition}

\begin{gather*}
\epsfig{file=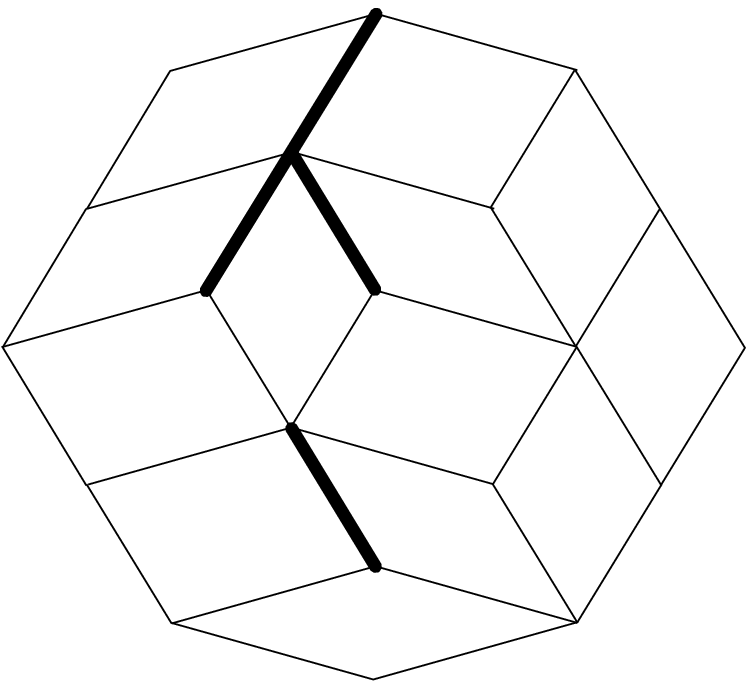, width=3cm}\\
\text{An example of a tiling and its fundamental forest.}
\end{gather*}

We will show in Lemma~\ref{lor} that $\Gamma$ is a forest, justifying its name.

\begin{lemma}\label{ars}\dontshow{ars}
Let $I\in T$ be a vertex, and $\alpha_1,\ldots,\alpha_r$ be the set of edges pointing down from $I$, ordered from left to right.
Then $I$ has edges of $\Gamma$ pointing down from it if and only if $r\ge 3$. 
These edges are then exactly $\alpha_2,\ldots,\alpha_{r-1}$.
\end{lemma}

\begin{proof}
Let $I_j$ denote the lower vertex of $\alpha_j$.
By definition, $\alpha_j\subset\Gamma$ if and only if $\alpha_j$ is internal and $I_j$ has exactly one edge pointing up from it.
We first show that $\alpha_1\not\in\Gamma$.
If $\alpha_1$ is a boundary edge, this is immediate since $\Gamma$ consists only of internal edges.
If $\alpha_1$ is an internal edge, then there exists a rhombus $R$ on its left.
Let $\beta$, $\beta'$ be the edges of $R$ incident to $I$, $I'$ respectively.
Since $\alpha_1$ is the leftmost edge pointing down, $\beta$ must be pointing up from $I$.
Since $\beta$ and $\beta'$ are parallel, $\beta'$ is also pointing up from $I_1$.
It follows that $I_1$ has more than one edge pointing up from it, and therefore that $\alpha_1\not\in\Gamma$.
Similarly, $\alpha_r$ is not in $\Gamma$.

\vspace{.3cm}{
\psfrag{i}{$I$}
\psfrag{i1}{$I_1$}
\psfrag{i2}{$I_2$}
\psfrag{i3}{$I_3\ldots$}
\psfrag{a1}{$\alpha_1$}
\psfrag{a2}{$\alpha_2$}
\psfrag{a3}{$\alpha_3$}
\psfrag{b}{$\beta$}
\psfrag{bp}{$\beta'$}
\psfrag{r}{$R$}
\centerline{\epsfig{file=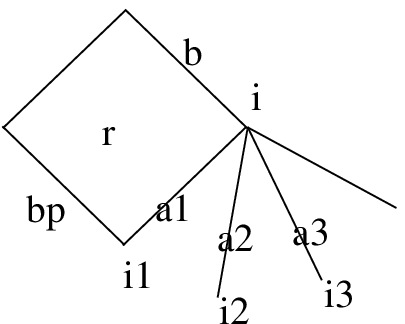, width=3cm}
\qquad\raisebox{20pt}{Two edges pointing up from $I_1$.}}
}\vspace{.3cm}

We now show that $\alpha_j\in\Gamma$ for $1<j<r$.
Since there are no edges coming down from $I$ between $\alpha_{j-1}$ and $\alpha_j$, there must be a rhombus adjacent to these two edges.
Similarly, there is a rhombus adjacent to $\alpha_j$ and $\alpha_{j+1}$.
An edge $(I_j,I')$, $I'\not=I$, pointing up from $I_j$ would intersect one of these two rhombi.
It follows that $\alpha_j$ is the only edge pointing up from $I_j$.
Hence $\alpha_j\subset\Gamma$.

\vspace{.3cm}{
\psfrag{i}{$I$}
\psfrag{ij}{$I_j$}
\psfrag{ijp}{$I_{j+1}$}
\psfrag{ijm}{$I_{j-1}$}
\psfrag{aj}{$\alpha_j$}
\psfrag{ajm}{$\alpha_{j-1}$}
\psfrag{ajp}{$\alpha_{j+1}$}
\psfrag{r}{}
\psfrag{rp}{}
\psfrag{ipp}{$I'$}
\centerline{\epsfig{file=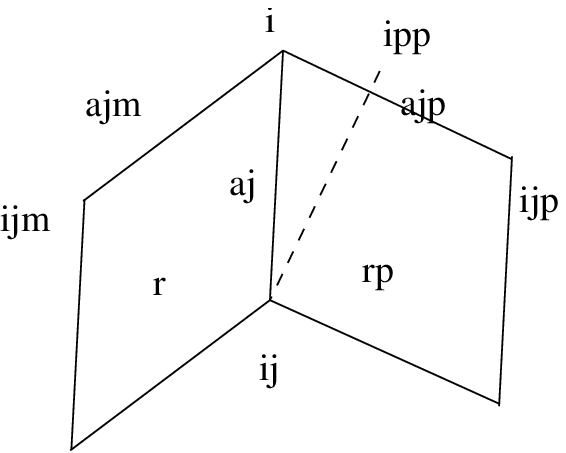, width=3cm}
\quad\raisebox{15pt}{$I'$ cannot be in $T$.}}
}\vspace{.1cm}
\end{proof}

\begin{lemma}\label{lor}\dontshow{lor}
%If we direct each edge of $\Gamma$ upward, then 
The graph $\Gamma$ is a rooted forest.
A vertex $I\in T$ is a leaf of $\Gamma$ if and only if it is an internal vertex of $T$ with two edges pointing down from it and one pointing up.
\end{lemma}

\begin{proof}
A vertex has an edge of $\Gamma$ pointing up from it if and only if it has exactly one edge pointing up from it. 
In particular, a vertex of $\Gamma$ has at most one edge of $\Gamma$ pointing up from it. 
Orient each edge in the direction of increasing $y$-coordinate.
Then $\Gamma$ has no directed cycles.
% as every edge points in the direction of increasing $y$-coordinate. 
A directed acyclic graph in which every vertex has out-degree at most $1$ is a rooted forest.

A leaf is a vertex $I$ with an edge of $\Gamma$ pointing up, and no edges of $\Gamma$ pointing down.
Having an edge of $\Gamma$ pointing up is equivalent to having exactly one edge of $T$ pointing up, and that edge being internal.
In other words, it is equivalent to $I$ being an internal vertex with exactly one edge of $T$ pointing up.
By Lemma~\ref{ars}, having no edges of $\Gamma$ pointing down is equivalent to having at most two edges of $T$ pointing down.
Since internal vertices have valence $\ge 3$, this means $I$ has exactly two edges pointing down.
\end{proof}

\begin{corollary}\label{sov}\dontshow{sov}
The set of vertices at which downward flips can be performed is equal to the set of leaves of $\Gamma$.\hfill $\square$
\end{corollary}

Given a fundamental forest $\Gamma$ and a vertex $v\in\Gamma$, we define the link of $v$ to be the set 
$\lk(v,\Gamma)$ of edges of $\Gamma$ pointing down from $v$.
If $v\not\in\Gamma$, then we let $\lk(v,\Gamma)=\emptyset$.

\begin{lemma}\label{unT}\dontshow{unT}
Let $T$, $T'$ be two distinct tilings of $P$, and let $\Gamma$, $\Gamma'$ be their respective fundamental forests.
Let $P_=$ be the image in $\RR^2$ of $T\cap T'$, and let $P_{\not =}$ be the closure of its complement.
Let $I\in P_{\not =}$ be a vertex with highest $y$-coordinate.
Then $\lk(I,\Gamma) \neq \lk(I,\Gamma')$.
\end{lemma}

\begin{proof}
Since $I$ is on the boundary of $P_{\not =}$ we also have $I\in P_=$\,.
Therefore $I$ belongs to both $T$ and $T'$.
Let $\alpha_1,\ldots,\alpha_r$ (respectively $\beta_1,\ldots,\beta_s$) be the set of edges of $T$ (resp. $T'$) pointing down from $I$, 
and suppose that they are ordered from left to right.
These sets are different because of our assumption on $I$.
However, we do have $\alpha_1=\beta_1$ and $\alpha_r=\beta_s$ because otherwise $I$ would not be the highest vertex of $P_{\not =}$.
It follows that $\{\alpha_2,\ldots,\alpha_{r-1}\}\not=\{\beta_2,\ldots,\beta_{s-1}\}$.
By Lemma~\ref{ars} the set of edges of $\Gamma$ (respectively $\Gamma'$) pointing down from $I$ is exactly $\alpha_2,\ldots,\alpha_{r-1}$ (resp. $\beta_2,\ldots,\beta_{s-1}$).
We have just established that these two sets are different. 
It follows that $\lk(I,\Gamma) \neq \lk(I,\Gamma')$.
\end{proof}

\begin{corollary}\label{to7}\dontshow{to7}
If $T\not =T'$ then $\Gamma\not =\Gamma'$.\hfill $\square$
\end{corollary}

%\begin{corollary}\label{tu9}\dontshow{tu9}
%Let $I$ be a point of highest $y$-coordinate among those with the property that $\lk(I,\Gamma) \neq\lk(I,\Gamma')$.
%Then $T$ and $T'$ are identical above $I$.
%\end{corollary}
%
%\begin{proof}
%Suppose that $T$ and $T'$ were not identical above $I$. 
%The vertex $I'\in P_{\not=}$ with highest $y$-coordinate would then be higher than $I$.
%By Lemma~\ref{unT}, it would have $\lk(I',\Gamma) \not = \lk(I',\Gamma')$, contradicting the maximality of $I$.
%\end{proof}

\begin{Proposition} \label{Kei}\dontshow{Kei} 
There is exactly one tiling $\Tmin$ on which no downward flip can be performed.
It is the unique tiling with empty fundamental forest.
\end{Proposition}

%\vspace{.3cm}
%\centerline{\epsfig{file=t1example.eps, width=3cm}\quad \raisebox{15pt}{An example of the tiling $\Tmin$.}}
%\vspace{.3cm}
				   
\begin{proof}
By Proposition~\ref{ncb}, there must be some tiling from which no downward flip is possible. 
By Corollary~\ref{sov}, if $T$ is a tiling from which no downward flip is possible, then the fundamental forest of $T$ is empty. 
But by Corollary~\ref{to7}, there is at most one tiling whose fundamental forest is empty.
So there is only one tiling from which no downward flip is possible.
\end{proof}

\begin{corollary}\label{cht}\dontshow{cht}
Let $T$ be a tiling and $T=T_0\rightsquigarrow T_1\rightsquigarrow T_2\rightsquigarrow\ldots$ be any sequence of downward flips. 
Then this process will eventually lead to $\Tmin$.
\end{corollary}

\begin{proof}
By Lemma~\ref{ncb}, the process of doing downward flips terminates.
By Proposition~\ref{Kei}, it can only terminate at $\Tmin$.
\end{proof}

We can give a geometric construction of $\Tmin$.
Since $\Gamma$ is empty, we know by Lemma~\ref{ars} that each vertex $I\in T$ can have at most two edges pointing down from it.
Starting from $P$, we build a tiling with the above property step by step.
Geometrically, that property means that when we see a vertex that has two edges pointing down from it, we must always fill that angle with a single rhombus. 
The empty polygon $P$ already has one such vertex, namely the top vertex $(a_1, \ldots, a_n)$. 
Filling this in and then filling in the rhombi that this forces \emph{etc}, it is easy to see that there is only one possibility.

\begin{gather*}
\psfrag{s}{$\Rightarrow$}
\psfrag{sd}{$\Rightarrow\,\,\ldots\,\,\Rightarrow$}
\epsfig{file=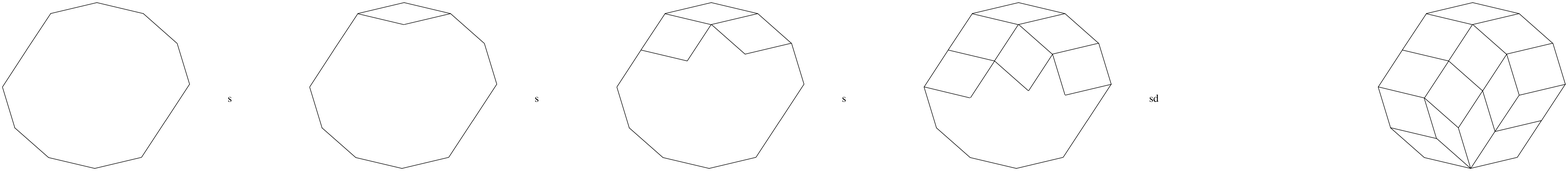, width=12cm} \\
\text{The construction of $\Tmin$ step by step.}
\end{gather*}

\begin{Proposition} \label{OTF}
The vertices of $\Tmin$ are those of the form
\dontshow{otf}
\begin{equation}\label{otf}
\hspace{1.8cm}(0,\ldots,0,b,a_{r+1},\ldots,a_{s-1},b',0,\ldots,0),\qquad \text{\small $1\!\le\! r\!\le\! s\!\le\! n,\,0\!<\! b\!\le\! a_r,\, 0\!<\! b'\!\le\! a_s$},
\end{equation}
along with the zero vector. Here, the numbers $a_i$ are those used in (\ref{dfP}).
\end{Proposition}

\begin{proof}
We first show that these vertices are indeed the vertices of a tiling, and then show that it admits no downward flips.

Put an edge between two vertices (\ref{otf}) if their difference is a unit vector.
For each internal edge $\alpha$, we need to verify that it has exactly two rhombi adjacent to it, one its right and one on its left.
Let $\alpha=(I, I')$ be such an edge, and suppose that $I$ is higher than $I'$.
Writing $I$ in the form (\ref{otf}), we then have $I'=I-e_r$ or $I'=I-e_s$. 
We treat the case $I'=I-e_r$, the other one being symmetric.
If $r=s$, then the two rhombi adjacent to $\alpha$ are $(\alpha,\alpha+e_{r-1})$ and $(\alpha,\alpha+e_{r+1})$.
If $r\not = s$, then we have two cases.
If $b'<a_s$, then the two rhombi are $(\alpha,\alpha-e_s)$ and $(\alpha,\alpha+e_s)$;
if $b'=a_s$, then they are $(\alpha,\alpha-e_s)$, $(\alpha,\alpha+e_{s+1})$.
So the vertices (\ref{otf}) are the vertices of a tiling. 

To see that this tiling has no downward flip, it is enough to check that its fundamental forest $\Gamma$ is empty.
There are at most two edges pointing up from any internal vertex (\ref{otf}): those pointing in the directions 
$e_r$ ($e_{r-1}$ if $b=a_r$) and $e_s$ ($e_{s+1}$ if $b'=a_s$ or $r=s$).
So by Lemma~\ref{ars}, we have $\Gamma=\emptyset$.
\end{proof}

Let $\XX_1$ be the graph whose vertices are the tilings of $P$, and whose edges are the flips.
As a direct consequence of Corollary~\ref{cht}, we obtain the following result.

\begin{Proposition}\label{cK}\dontshow{cK}
The graph $\XX_1$ is connected. \qed
\end{Proposition}

Let $\XX_2$ be the cellular complex obtained by gluing the following squares and octagons onto $\XX_1$:
 
{\it i)}
If $T\rightsquigarrow T'$ and $T\rightsquigarrow T''$ are flips involving disjoint sets of rhombi, 
we can perform the two of them simultaneously to get a fourth tiling $T'''$.
The vertices $T,T',T'',T'''$ form a 4-cycle in $\XX_1$ on which we attach a square.

{\it ii)} 
Suppose that $\sigma$ is a four-cell of $C$, and that $T$ a tiling containing a tiling of $\pi(\sigma)$.
Then $T$ contains one of the figures in (\ref{figoc}) as a subset and we may perform the corresponding cycle of eight flips. 
In each such case, we glue an octagon with boundary this series of eight flips.

\begin{Proposition} \label{Key} \dontshow{Key} 
The cell complex $\XX_2$ is simply connected.
\end{Proposition}

\begin{proof}
Let $\phi$ be the functional defined in the proof of Lemma~\ref{ncb}.
Recall that for upward flips $T\rightsquigarrow T'$, we have $\phi(T')=\phi(T)+1$, and inversely for downward flips.
Let $\gamma=(T_1,T_2,\ldots,T_m=T_1)$ be a loop in $\XX_1$.
We'll prove by induction on $\max(\phi(T_i))$ that $\gamma$ is nullhomotopic in $\XX_2$. Let $i$ be an index at which $\phi(T_i)$ is maximized.

Let $T_{i,0}:=T_i$, and inductively define $T_{i,j}$ by the following procedure.
First pick a leaf of the fundamental forest of $T_{i,j-1}$ with smallest possible $y$-coordinate.
Then let $T_{i,j}$ be the tiling obtained by performing a downward flip at that vertex.
By Corollary~\ref{cht}, this process terminates at $\Tmin$.
Let $\gamma_i=(T_{i,0},T_{i,1},\ldots)$ be the resulting path, and let $\alpha_i$ denote the edge $(T_i,T_{i+1})$ of $\XX_1$.
If $D_i$ are discs in $\XX_2$ filling the loops $\alpha_i\gamma_{i+1}\gamma_i^{-1}$, then $\cup D_i$ bounds $\gamma$.
So we have reduced ourselves to showing that $\alpha_i\gamma_{i+1}\gamma_i^{-1}$ is nullhomotopic.

If $\alpha_i$ is a downward flip, let $\delta:=\alpha_i\gamma_{i+1}$, $\delta':=\gamma_i$, $S:=T_i$, otherwise let $\delta:=\alpha_i^{-1}\gamma_i$, $\delta':=\gamma_{i+1}$, $S:=T_{i+1}$.
Note that $\delta$, $\delta'$ are paths from $S$ to $\Tmin$ and that they consist entirely of downward flips.
Let $\Gamma$ be the fundamental forest of $S$.
Up to reparametrization, the loop $\delta^{-1}\delta'$ is equal to $\alpha_i\gamma_{i+1}\gamma_i^{-1}$. 
So we need to show that $\delta^{-1}\delta'$ is nullhomotopic.

Write $\delta=(S,S_1,S_2,\ldots, S_r=\Tmin)$, $\delta'=(S,S'_1,S'_2,\ldots, S_r=\Tmin)$, and
let $\beta$, $\beta'$ denote the flips $S\rightsquigarrow S_1$, $S\rightsquigarrow S'_1$ respectively.
The sequence of flips in $\delta^{-1}\delta'$ looks like this: 
\[
\Tmin
\rightsquigarrow\,\,
\ldots \,\,{\rightsquigarrow}
\,S_2\,{\rightsquigarrow}\,S_1\,
\stackrel{\beta^{-1}}{\rightsquigarrow}
\,S\,
\stackrel{\beta'}{\rightsquigarrow}
\,S'_1\,{\rightsquigarrow}\,S'_2\,{\rightsquigarrow}\,\,\ldots
\,\,\rightsquigarrow
\Tmin
\]
Let $I,I'\in S$ be the vertices at which $\beta$ and $\beta'$ are performed.
We know by construction that the vertex $I'$ is a leaf of $\Gamma$ with smallest $y$-coordinate.

If $\beta$ and $\beta'$ involve disjoint sets of rhombi, we can replace $\beta^{-1},\beta'$ by the other two sides of the 4-cycle {\it (i)}. 
This produces a loop whose maximum value of $\phi$ is $r-1$, and we're done by induction.
Otherwise, the supports of $\beta$ and $\beta'$ overlap, and $S$ therefore looks like this: \dontshow{figov}
\begin{equation}\label{figov}
\psfrag{I}{$I$}
\psfrag{Ip}{$I'$}
\begin{matrix}\psfig{file=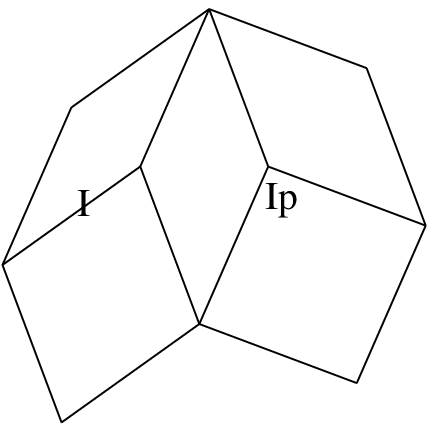, width=2cm}\end{matrix}
\end{equation}
The vertex just below $I$ and $I'$ has only two edges pointing down from it, since
otherwise by Lemma~\ref{ars}, it would have edges of $\Gamma$ pointing down from it, contradicting the minimality of $I'$.
So figure (\ref{figov}) fits into an octagon similar to the one in the upper left of (\ref{figoc}). 
The flips $\beta^{-1}$ and $\beta'$ are two of the flips in the 8-cycle (\ref{figoc}). Replace them by the six remaining flips.
This produces a loop whose maximum value of $\phi$ is $r-1$, and we're again done by induction.
\end{proof}

We will need the following technical variant of Proposition~\ref{cK}.

\begin{Proposition} \label{Qus} \dontshow{Qus}
Let $T$, $T'$ be two tilings of $P$, and let $I_0 \in \Pi$ be a vertex contained in both $T$ and $T'$.
Then there is a sequence of flips going from $T$ to $T'$, so that all intermediate tilings contain $I_0$.
\end{Proposition}

\begin{proof}
Let $\mathcal T$ be the set of tilings containing $I_0$, and let $\mathcal T_0\subset\mathcal T$ be the 
subset consisting of those tilings on which no downward flip can be done, 
except possibly at the vertex $I_0$.
By Lemma~\ref{ncb}, any tiling in $\mathcal T$ can be joined to one in $\mathcal T_0$ by a sequence of downward flips in such a way that all intermediate tilings contain $I_0$. 
We may therefore assume that $T$ and $T'$ are in $\mathcal T_0$.

Let $T\in\mathcal T_0$ be a tiling.
The set of vertices on which downward flips are possible is either empty or equal to $\{I_0\}$.
If it is empty, let $I':=I_0$.
Otherwise, we know by Corollary~\ref{sov} that $I_0$ is the unique leaf of $\Gamma$.
It follows that $\Gamma$ is just a path.
Let $I'$ then denote the root of $\Gamma$.

\vspace{.3cm}
\centerline{
\psfrag{IO}{$I_0$}
\psfrag{II}{$I'$}
\epsfig{file=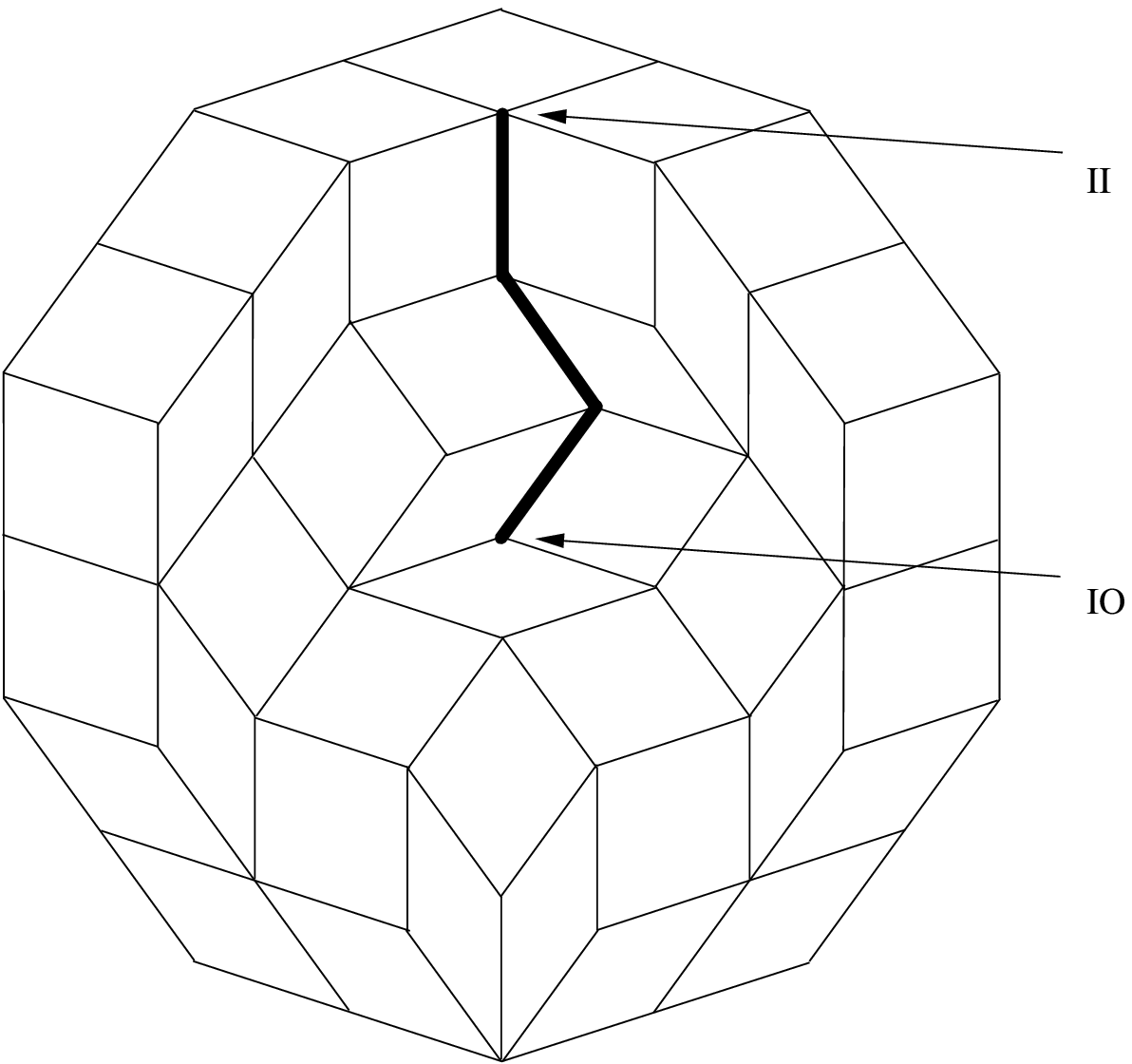, width=4cm}\quad \raisebox{15pt}{An example of $T\in\mathcal T_0$.}}
\vspace{.3cm}

We claim that $I'$ only depends on $I_0$, and not on the particular tiling $T\in\mathcal T_0$. 
To see this, write $I_0=(0,\ldots,0,i_r,i_{r+1},\ldots, i_s,0,\ldots,0)$, 
where $r$ and $s$ are the first and last nonzero coordinates.
The fundamental forest of $\Tmin$ being empty,
it follows from Lemma~\ref{unT} that $T$ and $\Tmin$ are identical above $I'$.

Indeed, if $T$ and $\Tmin$ were not identical above $I'$, the vertex $\tilde I\in P_{\not =}$ with highest $y$-coordinate would be higher than $I'$. By Lemma~\ref{unT}, we would have $\lk(\tilde I,\emptyset)\not =\lk(\tilde I,\Gamma)$, which is impossible since both links are empty.

The above argument shows in particular that $I'$ is a vertex of $\Tmin$. 
So by Proposition~\ref{OTF}, it is the form
\begin{equation}\label{hwo}
I'=(0,\ldots,0,b,a_{\tilde r+1},\ldots,a_{\tilde s-1},b',0,\ldots,0).
\end{equation}
Since $\Gamma$ is a path, we know by Lemma~\ref{ars} that every vertex $I\in\Gamma$, $I\not=I_0$, has exactly three edges pointing down from it: 
two outer edges which are not in $\Gamma$, and one edge in the middle that belongs to $\Gamma$.
Clearly, the directions of the two outer ones do not depend on the particular vertex of $\Gamma$.
Let us call those directions $-e_{\hat r}$ and $-e_{\hat s}$.
The two edges pointing down from $I_0$ also point in the directions $-e_{\hat r}$ and $-e_{\hat s}$.
Since $I_0-e_{\hat r}$ and $I_0-e_{\hat s}$ are elements of $\Pi$, we must have $r\le \hat r$ and $\hat s\le s$ 
(in fact, we have $\hat r=r$ and $\hat s= s$, but this fact will not be needed here).
The edges of $\Gamma$ thus point in directions $-e_j$ for various $j\in[\hat r+1,\hat s-1]\subset [r+1,s-1]$
and, in particular,
the first and last non-zero coordinates of $I_0$ and $I'$ agree.
In equation (\ref{hwo}), we therefore have $\tilde r=r$, $\tilde s=s$, $b=i_r$, and $b'=i_s$,
from which we get 
\[
I'=(0,\ldots,0,i_r,a_{r+1},\ldots, a_{s-1}, i_s,0,\ldots,0),
\]
which is indeed a formula independent of $T$.

Let $I_0, I_1, \ldots, I_m=I'$ denote the set of vertices of $\Gamma$, ordered from the leaf to the root.
We then have $I_n=I_{n-1}+e_{j_n}$ for various $j_n\in[r+1,s-1]$,
and the sequence $J=(j_1, j_2,\ldots)$ contains each element $e_j$ exactly $a_j-i_j$ times.
Let \dontshow{Jpr}
\begin{equation}\label{Jpr}
J'=(j_1,\ldots,j_{n-1},j_{n+1},j_n, j_{n+2},\ldots)
\end{equation}
be another such sequence, obtained from $J$ by a transposition.
We show that $J'$ corresponds to a tiling $T'\in\mathcal T_0$, 
and that $T'$ can be obtained from $T$ by a sequence of flips in such a way that all intermediate tilings contain $I_0$.

Let $R_0$ be the rhombus to the left of the edge $(I_n,I_{n+1})$ and $R_1$ be the one to its right.
Inductively define $R_i$ to be the rhombus adjacent to $R_{i-1}$, opposite from $R_{i-2}$.
If $i<0$, we let $R_i$ be the rhombus adjacent to $R_{i+1}$, opposite from $R_{i+2}$.
Let $\mathcal R=\bigcup R_i$. 
Similarly, let $\mathcal S=\bigcup S_i$ be the chain of rhombi\footnote{These chains of rhombi corresponds to pseudolines in the pseudoline arrangement for $T$.} 
corresponding to the edge $(I_{n-1},I_n)$.
Since $\mathcal R$ and $\mathcal S$ link different pairs of opposite edges of the polygon $P$, there exists a rhombus $V\subset\mathcal R\cap \mathcal S$.
Let us assume that $V$ lies on the left side of $\Gamma$ (the other case being symmetric).

\vspace{.3cm}
\centerline{
\psfrag{R}{$V$}
\psfrag{G}{$\Gamma$}
\psfrag{cR}{$\mathcal R$}
\psfrag{cS}{$\mathcal S$}
\epsfig{file=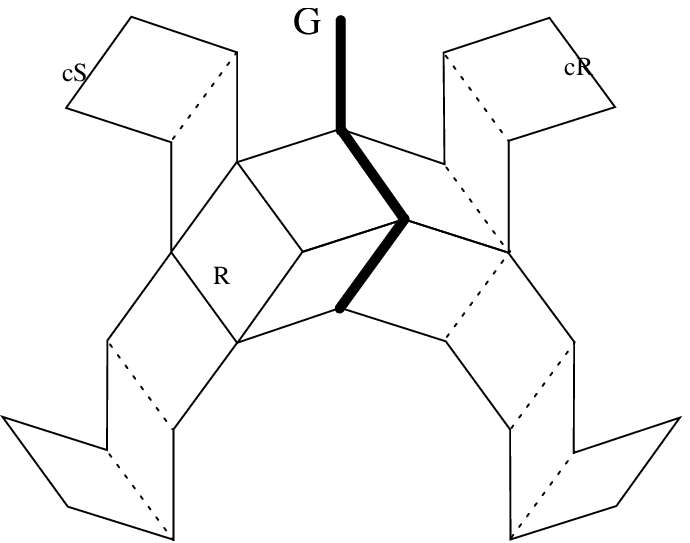, width=4cm}}
\vspace{.3cm}

\noindent
We then have $V=R_a=S_b$ for some $a,b\le 0$.
Let $\alpha_a$ be the edge between $V$ and $S_{b+1}$. 
For $i>a$, define inductively $\alpha_i$ to be the edge of $R_i$ which is in the boundary of $\mathcal R$,
and which is adjacent to $\alpha_{i-1}$.
Let also $X_i$ be the vertex between $\alpha_i$ and $\alpha_{i+1}$.
For $a<i\le 0$, the edges $\alpha_i$ point up from $X_{i-1}$, 
as otherwise the edge between $R_{i-1}$ and $R_i$ would be the only one pointing up from $X_{i-1}$ and thus be in $\Gamma$.
Since $(I_n,I_{n+1})$ is in $\Gamma$, it is the only edge pointing up from $X_0=I_n$, and therefore $\alpha_1$ points down from it.
Let $c>0$ be the biggest number such that $\alpha_i$ points down from $X_{i-1}$ for all $1\le i\le c$.
By induction on $i$, we see that $S_i$ touches $R_i$ for all $0<i\le c$ since otherwise there would be too many edges pointing down from $X_{i-1}$, thus contradicting Lemma~\ref{ars}.
Similarly, $S_i$ touches $R_i$ for all $a<i\le 0$ since otherwise there would be too many edges pointing down from $X_i$ (in particular $a=b$). 
The following figure illustrates the preceding definitions:

\begin{equation*}
\begin{matrix}
\psfrag{V}{$V$}
\psfrag{rm2}{$\scriptstyle R_{\text{-}\!2}$}
\psfrag{rm1}{$\scriptstyle R_{\text{-}\!1}$}
\psfrag{r0} {$\scriptstyle R_0$}
\psfrag{r1} {$\scriptstyle R_1$}
\psfrag{r2} {$\scriptstyle R_2$}
\psfrag{r3} {$\scriptstyle R_3$}
\psfrag{r4} {$\scriptstyle R_4$}
\psfrag{sm2}{$\scriptstyle S_{\text{-}\!2}$}
\psfrag{sm1}{$\scriptstyle S_{\text{-}\!1}$}
\psfrag{s0} {$\scriptstyle S_0$}
\psfrag{s1} {$\scriptstyle S_1$}
\psfrag{s2} {$\scriptstyle S_2$}
\psfrag{s3} {$\scriptstyle S_3$}
\psfrag{xm3}{$\scriptstyle X_{\text{-}3}$}
\psfrag{xm2}{$\scriptstyle X_{\text{-}\!2}$}
\psfrag{xm1}{$\scriptstyle X_{\text{-}\!1}$}
\psfrag{x0} {$\scriptstyle X_0$}
\psfrag{x1} {$\scriptstyle X_1$}
\psfrag{x2} {$\scriptstyle X_2$}
\psfrag{x3} {$\scriptstyle X_3$}
\psfrag{am3}{$\scriptstyle \alpha_{\text{-}3}$}
\psfrag{am2}{$\scriptstyle \alpha_{\text{-}\!2}$}
\psfrag{am1}{$\scriptstyle \alpha_{\text{-}\!1}$}
\psfrag{a0} {$\scriptstyle \alpha_0$}
\psfrag{a1} {$\scriptstyle \alpha_1$}
\psfrag{a2} {$\scriptstyle \alpha_2$}
\psfrag{a3} {$\scriptstyle \alpha_3$}
\epsfig{file=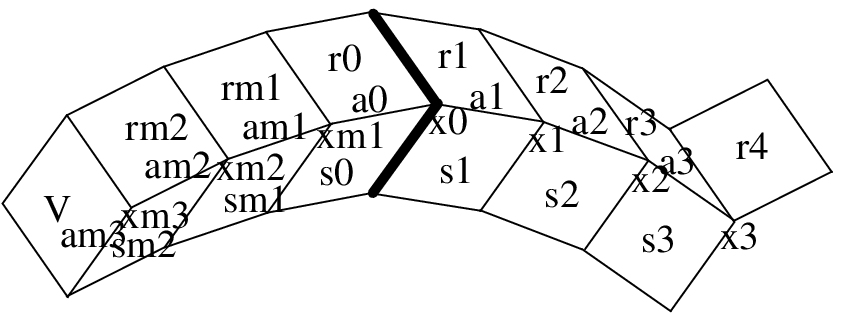, width=8cm} 
\end{matrix} %\label{RSDiagram}
\end{equation*}

\noindent
In this figure, $a=b=-3$, and $c=3$.

Since $S_i$ and $R_i$ are adjacent for $a<i\le c$, we can perform flips at the vertices $X_a$, $X_{a+1}$, $\ldots$ , $X_{c-1}$.
Let $T'$ be resulting tiling and $\Gamma'$ its fundamental forest.
It is easy to see that $\Gamma'$ is the path obtained from $\Gamma$ by exchanging the edges $(I_{n-1},I_n)$ and $(I_n,I_{n+1})$.
So $T'$ is the desired tiling corresponding to (\ref{Jpr}). 
Moreover, it is constructed from $T$ by a sequence of flips such that all intermediate tilings contain $I_0$.

\vspace{.3cm}
\centerline{
\psfrag{s}{$\rightsquigarrow$}
\psfrag{t}{$T=$}
\psfrag{tp}{$=T'$}
\psfrag{g}{$\Gamma$}
\psfrag{gp}{$\Gamma'$}
\psfrag{sds}{$\rightsquigarrow\ldots\rightsquigarrow$}
\epsfig{file=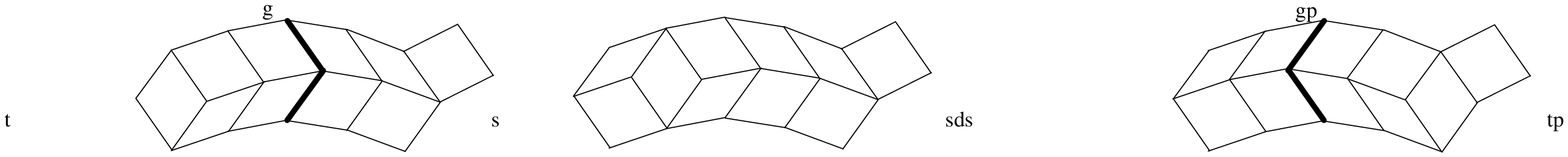, width=12cm}}
\vspace{.3cm}

Suppose now that we have $T, T'\in\mathcal T_0$,
and let $J,J'$ be the corresponding sequences, as in (\ref{Jpr}).
Since $J$ is a permutation of $J'$, we may pick intermediate sequences $J=J_0,J_1,\ldots,J_t=J'$, such that each $J_i$ differs from $J_{i-1}$ by a transposition.
By the above argument, we have tilings $T_i\in\mathcal T_0$ associated to $J_i$, and each $T_i$ is obtained from $T_{i-1}$ 
by a sequence of flips such that all intermediate tilings contain $I_0$.
The last tiling $T_t$ has the same fundamental forest as $T'$, so by Corollary~\ref{to7}, we have $T_t=T'$.
Putting all this together, we have produced a way of going from $T$ to $T'$ by a sequence of flips such that all intermediate tilings contain $I_0$.
\end{proof}

\section{Transcendence Bases and Laurentness\dontshow{s:L}}\label{s:L}

Let $T_0$ be a tiling of $P$, and let $x_J$, $J\in\vert(T_0)$, be formal variables.
By Theorem~\ref{h:i}, we know that we can label the rest of $\Pi\setminus\vert(T_0)$ with elements $x_I\in\QQ(\{x_J\}_{J\in\vert(T_0)})$,
in such a way that all the equations (\ref{trc}) are satisfied.

\begin{Proposition}\label{yo}\dontshow{yo} 
Let $T_0$ be a tiling of $P$
and let $I\in \Pi$ be a point.
Then $x_I$ is a Laurent polynomial in the variables $\{x_J\}_{J\in \vert (T_0)}$.
\end{Proposition}

\begin{proof}
%We first show that $x_I $ is
%given by a rational expression in the $x_J$. 
%By Lemma~\ref{LA}, there is 
%
%Each time we make a flip, the cube recurrence allows us to write the new variable as a rational expression in
%the old variables.
%So, assuming that we never have to divide by zero, the variable $x_I $ is a rational function of the $x_J$'s.
%
%To see that we do not encounter polynomials that are identically zero, we can embed the field
%$\QQ(\{x_J\}_{J\in \vert (T_0)})$ in $\RR$ by sending $x_J$ to (algebraically independent) positive real numbers.
%After each step of the recurrence, the new value is again represented by a positive real number. 
%In particular, it is a non-zero expressions of the $x_J$'s.

We apply Theorem 2.1 of~\cite{FZ} (known as the caterpillar lemma), in a manner very
similar to Fomin and Zelevinsky's proof of the case $n=3$~\cite[Theorem 1.2]{FZ}. 
Our notation matches~\cite{FZ} as much as possible.

Pick a tiling $T'$ containing $I$, and a sequence of flips 
\[
T_0\rightsquigarrow T_1\rightsquigarrow\,\,\ldots\,\,\rightsquigarrow T_N=T'
\]
going from $T_0$ to $T'$. 
Let $V:=\vert(T_0)$.
For each $k$, place the vertices of $T_k$ in bijection with the vertices of $T_{k+1}$ as follows. 
For each $k$, there is exactly one vertex of $T_k$ which is not a vertex of $T_{k+1}$ and vice versa. 
Pair these two vertices with each other and let every vertex which is in both $T_k$ and $T_{k+1}$ be paired with itself. 
Thus, the vertices of each $T_i$ are naturally labeled by $V$. 
We now create the ``caterpillar tree", which we denote $\Cater$. 
It contains a path with $N+1$ vertices, which we will denote $t_0$, \dots, $t_N$. 
Every edge $e$ of $\Cater$ is labeled with an element $v(e)$ of $V$; specifically, 
the edge between $t_k$ and $t_{k+1}$ is labeled with the element of $V$ corresponding to the vertex of $T_{k}$ not in $T_{k+1}$. 
Also, $\Cater$ has $2(|V|-1)+(N-1)(|V|-2)$ additional vertices, each of degree one. 
There are $|V|-1$ of these additional vertices adjacent to $t_0$ and $t_N$ and $|V|-2$ adjacent to each other $t_i$. 
The edges to these additional vertices are labeled with elements of $V$ in such a way 
that each element of $V$ occurs exactly once among the labels of the edges adjoining each $t_i$.

\begin{gather}
\psfrag{t0}{$t_0$}
\psfrag{t1}{$t_1$}
\psfrag{t2}{$t_2$}
\psfrag{t3}{$t_3$}
\psfrag{tn-1}{$t_{N-1}$}
\psfrag{tn}{$t_N$}
\psfrag{dots}{$\scriptstyle .\,\,.\,\,.$}
\psfrag{v-2}{\put(0,8){\rotatebox{90}{$\Big\{$}}$\scriptstyle|V|-1$}
\psfrag{v-1}{\put(1,8){\rotatebox{90}{$\Big\{$}}$\scriptstyle|V|-2$}
\nonumber\psfig{file=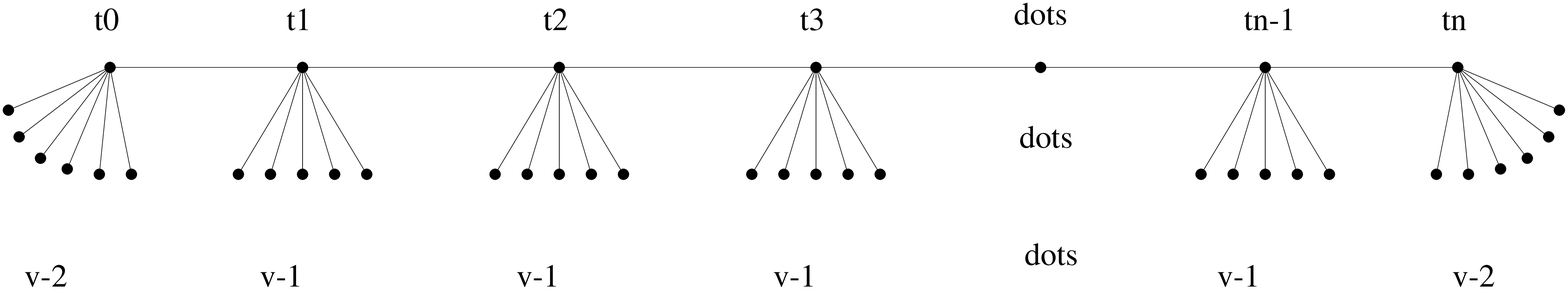, width=13cm}\\
\nonumber\text{The caterpillar tree $\Cater$.}
\end{gather}

Consider the polynomial ring $\ZZ[x_v]_{v \in V}$, whose variables are indexed by the elements $v$ of $V$. 
To each edge $e$ of $\Cater$, we assign an element $p(e)$ of that ring as follows. 
At least one of the endpoints of $e$ is of the form $t_i$; 
if both are, choose one arbitrarily. 
Now, $v(e)$ corresponds to a vertex of the tiling $T_i$ of $P$. 
Let $a_1, \dots, a_r,a_{r+1}=a_1$, be the neighbors of $v(e)$ in this tiling, written in cyclic order.
For $j=1\,.\,.\,r$, let $b_{j}$ be the fourth vertex of the rhombus containing $v(e)$, $a_j$, and $a_{j+1}$.
If $v(e)\in \partial P$, then one of the $b_j$ will fail to exist: in that case, let\footnote{In the main part of our argument, where we check Condition~3, it would suffice to take any value whatsoever for $x_{b_j}$ in the case where $v(e) \in \partial(P)$, as long as our choice only depended on $v(e)$. However, in our check of Conditions~1 and~2, we need to make sure that the polynomials $p(e)$ don't acquire any spurious factors that would interfere with the argument.  The expression $x_{a_j} x_{a_{j+1}}$ was the simplest choice we could find which avoided this issue, there are doubtless many others.} $x_{b_j}:=x_{a_j} x_{a_{j+1}}$.

%
%is defined for $j=1$, $2$, \dots, $r-1$; if $v(e)$ is in the interior of $P$ then $b_r$ is also defined as the fourth vertex of the rhombus containing $v(e)$, $a_1$ and $a_r$. We describe our argument in the case that $v(e)$ is in the interior but only minor modifications are needed for $v(e)$ on the boundary.) Label the edge $e$ of $\Cater$ with the polynomial 
%
We now define $p(e)\in \ZZ[x_v]$ by the formula
\begin{equation}
p(e) := x_{a_1} x_{a_2} \cdots x_{a_r} \left( \frac{x_{b_1}}{x_{a_1} x_{a_2}} + \frac{x_{b_2}}{x_{a_2} x_{a_3}} + \cdots + \frac{x_{b_r}}{x_{a_r} x_{a_1}} \right) \label{EdgeLabel},
\end{equation}
and note that, if both ends of $e$ are of the form $t_i$, then $p(e)$ does not depend on which end we picked. 
Carrying out the cube recurrence consists of traveling from $t_0$ to $t_N$ and
replacing the variable $x_{v(e)}$ by the expression $p(e)(x_1, \ldots, x_{|V|})/x_{v(e)}$ each time we travel across an edge $e$.
According to Theorem~2.1 of~\cite{FZ}, the resulting rational expressions will be Laurent polynomials in the $x_v$ provided we check the three conditions below.

\textbf{Condition 1:} For every edge $e$, the polynomial $p(e)$ does not depend on $x_{v(e)}$ and is not divisible by $x_v$ for any $v \in V$. This is clear by inspection. 

\textbf{Condition 2:} Let $e$ and $e'$ be two edges of $\Cater$ bordering the same vertex $t_i$. 
Then $p(e)$ and $p(e')|_{v(e)=0}$ are relatively prime in $\ZZ[x_v]$. 
In fact, both of these polynomials are irreducible.
Indeed, their Newton polytopes are simplices which are not integral multiples of smaller simplices, 
and thus cannot be expressed as nontrivial Minkowski sums. 
The various Newton polytopes are not equal to each other, so the polynomials are relatively prime.

\textbf{Condition 3:} Consider a four vertex sub-chain of $\Cater$ with edges $e_1$, $e_2$, and $e_3$, and let us assume that $v(e_1)=v(e_3)$. 
We adopt the shorthand $P=p(e_1)$, $Q=p(e_2)$, $R=p(e_3)$, $i=v(e_1)=v(e_3)$ and $j=v(e_2)$. 
Let $t_k$ and $t_{k+1}$ be the endpoints of $e_2$. 
The condition is that $R|_{x_j \leftarrow \frac{Q_0}{x_j}}$ is of the form $L Q_0^m P$ where $L$ is a Laurent monomial coprime to $P$, $Q_0$ is $Q|_{x_i=0}$ and $b$ is a nonnegative integer. 
Here $R|_{x \leftarrow g(x,y, \ldots)}$ denotes the result of replacing $x$ by $g(x,y,\ldots)$ in the expression for $R$. 
There are three cases:

\textbf{Case 1:}
$i$ and $j$ are not in the same rhombus. In this case, $x_j$ does not occur in $R$ and $R=P$ so the condition is trivially satisfied with $L=1$ and $m=0$.

\textbf{Case 2:}
$i$ and $j$ are joined by an edge in $T_k$. This implies that $i$ and $j$ are diagonally opposite vertices of the same rhombus in $T_{k+1}$ and that, in both tilings, $j$ has degree $3$. Let $a_1$, \dots, $a_r$ be the neighbors of $i$ in $T_k$ other than $j$. Let the rhombi of $T_k$ containing $i$ be $(i,j,c, a_1)$, $(i, a_1, b_1, a_2)$, $(i,a_2,b_2,a_3)$, \dots, $(i,a_{r-1}, b_{r-1}, a_r)$, $(i,a_r,e,j)$, and let $(j,e,d,c)$ be the remaining rhombus of $T_k$ containing $j$
(if $i\in\partial P$, then one of the above rhombi doesn't actually exist, but this doesn't affect the rest of the argument).
%The reader may want to consult figure (\ref{Tk}).

\[
\psfrag{c}{$\scriptstyle c$}
\psfrag{d}{$\scriptstyle d$}
\psfrag{e}{$\scriptstyle e$}
\psfrag{i}{$\scriptstyle i$}
\psfrag{j}{$\scriptstyle j$}
\psfrag{a1}{$\scriptstyle a_1$}
\psfrag{a2}{$\scriptstyle a_2$}
\psfrag{a3}{$\scriptstyle a_3$}
\psfrag{a4}{$\scriptstyle a_4$}
\psfrag{a5}{$\scriptstyle a_5$}
\psfrag{a6}{$\scriptstyle a_6$}
\psfrag{b1}{$\scriptstyle b_1$}
\psfrag{b2}{$\scriptstyle b_2$}
\psfrag{b3}{$\scriptstyle b_3$}
\psfrag{b4}{$\scriptstyle b_4$}
\psfrag{b5}{$\scriptstyle b_5$}
\psfrag{Tk}{The tiling $T_{k}$}
\psfrag{Tk+1}{The tiling $T_{k+1}$}
\psfig{file=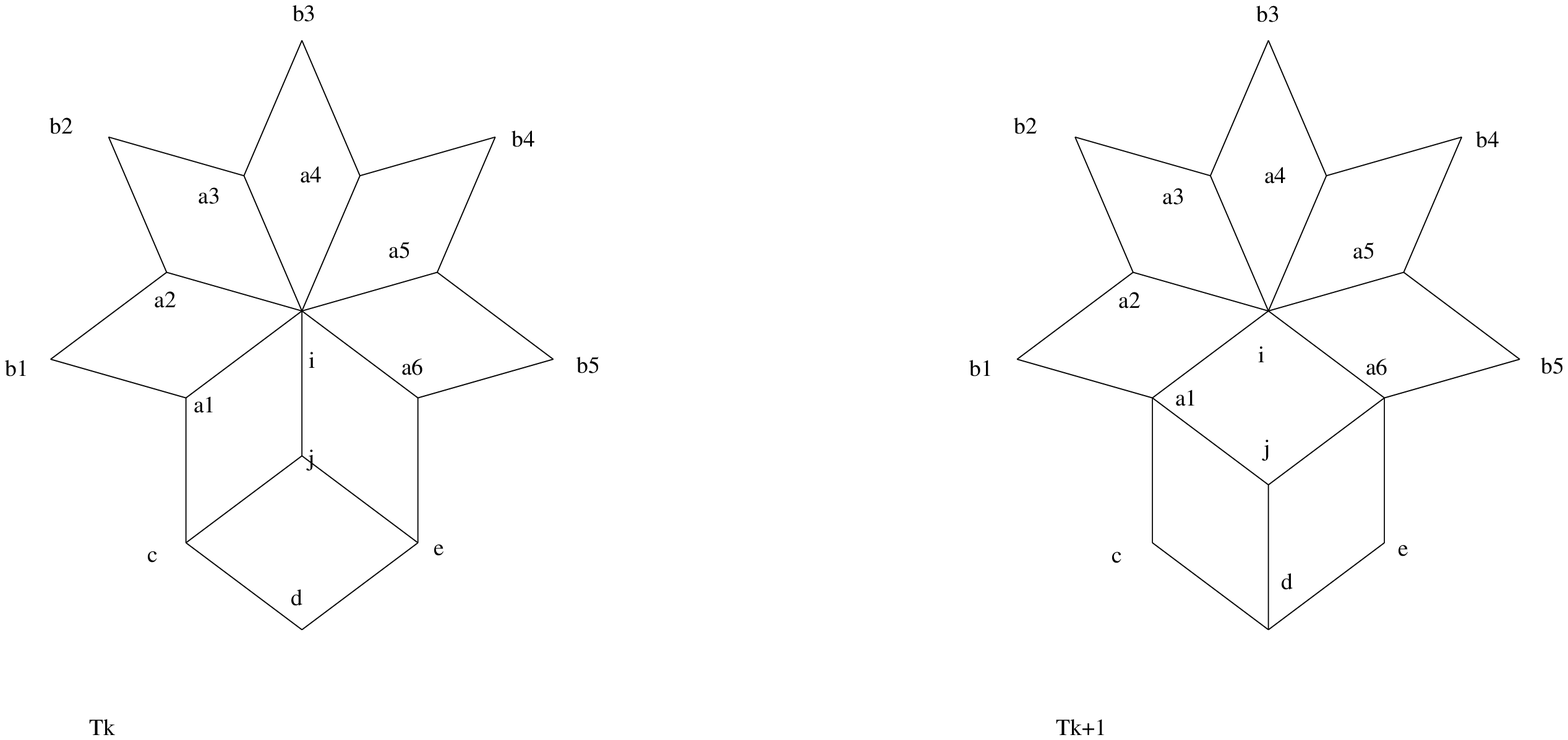, width=9cm}
\]

%\begin{gather*}
%YYY \\
%\epsfig{file=Tk.eps, width=6cm}  \\
%\text{The tilings $T_{k}$ and $T_{k+1}$}
%\end{gather*}

Then we have
\begin{equation}
\begin{split}
P\hspace{.13cm}   &=\, x_{a_1} \cdots x_{a_r} x_j \left( \frac{x_c}{x_{j} x_{a_1}} + \frac{x_e}{x_{a_r} x_{j}} + U \right) \\
Q\hspace{.13cm}   &=\, x_i x_d + x_c x_{a_r} + x_e {x_{a_1}}_{\phantom{\big |}} \\ 
Q_0              &=\, x_c x_{a_r} + x_e x_{a_1} \\
R\hspace{.13cm}   &=\, x_{a_1} \cdots x_{a_r} \left( \frac{x_j}{x_{a_1} x_{a_r}} + U \right)
\end{split}
\end{equation}
where $U=\sum_{j=1}^{r-1} \frac{b_j}{a_j a_{j+1}}$ is independent of $x_j$. Using the identity
\[
\frac{x_c}{x_j x_{a_1}} \,+\, \frac{x_e}{x_{a_r} x_j}\, =\, \frac{\hspace{5mm}\frac{x_c x_{a_r} + x_e x_{a_1}}{x_j}\hspace{5mm}}{x_{a_1}x_{a_r}}\,\,,
\]
we find that the required identity holds with $m=0$ and $L=x_j^{-1}$. 

\textbf{Case 3:}
$i$ and $j$ are diagonally opposite members of the same rhombus in $T_{k}$. This is like case $2$ with the roles of $T_k$ and $T_{k+1}$ switched. We have
\begin{equation}
\begin{split}
P\hspace{.13cm} &=\, x_{a_1} \cdots x_{a_r} \left( \frac{x_j}{x_{a_1} x_{a_r}} + U \right) \\
Q\hspace{.13cm} &=\, x_i x_d + x_c x_{a_r} + x_e {x_{a_1}}_{\phantom{\big |}} \\
Q_0 &=\, x_c x_{a_r} + x_e x_{a_1} \\
R\hspace{.13cm} &=\, x_{a_1} \cdots x_{a_r} x_j \left( \frac{x_c}{x_{j} x_{a_1}} + \frac{x_e}{x_{a_r} x_{j}} + U \right)
\end{split}
\end{equation}
where $U$ is the same as in the previous case. Using the identity
\[
\frac{x_c}{\hspace{3mm}\frac{x_c x_{a_r} + x_e x_{a_1}}{x_j}\, x_{a_1}\hspace{0mm}} \,+\, \frac{x_e}{\hspace{2mm}x_{a_r} \frac{x_c x_{a_r} + x_e x_{a_1}}{x_j}\hspace{2mm}}\,=\,\frac{x_j}{x_{a_1} x_{a_r}}\,,
\] 
we see that the required identity holds with $m=1$ and $L=1$. 
\end{proof}

\begin{lemma}\label{fw}\dontshow{fw}
Let $T$ be a tiling of $P$.
Then $T$ has $\sum_{ i < j } a_i a_j+\sum_i a_i+1$ vertices
\end{lemma}

\begin{proof}
Let $V, E, F$ denote the numbers of vertices, edges, and faces of $T$.
Under the correspondence between tilings and pseudo-line arrangements~\cite{Zieg}, 
the tilings of $P$ correspond to arrangements of $\sum a_i$ lines that come in $n$ families of ``parallel'' pseudolines, 
the $i^{\textrm{th}}$ family containing $a_i$ of them.
These pseudolines intersect in $\sum_{i<j} a_i a_j$ points.
Taking planar duals, we conclude that $F=\sum_{i<j} a_i a_j$.

Every rhombus has four edges, and every edge is contained in two rhombi except for the $2 \sum a_i$ boundary edges, 
which are contained in only one rhombus. 
It follows that $4F=2E-2 \sum a_i$. 
Since the Euler characteristic of $P$ is one, we have 
\begin{equation*}
%\tag*{\put(-9,0){$\square$}}
V=E-F+1=(2F+\sum a_i)-F+1=F+\sum a_i+1=\sum_{ i < j } a_i a_j+\sum a_i+1.\\
\end{equation*} 
\end{proof}
%\vspace{0cm}

We now can describe the basic geometry of the variety $X=X(A)$, defined in the introduction.

\begin{theorem}\label{irred}\dontshow{irred}
The scheme $\Yo=\Yo(A)\subset (\CC^\times)^{\Pi}$ cut out by (\ref{trc}) is an irreducible variety of dimension $\sum_{ i < j } a_i a_j+\sum_i a_i+1$.
If $T$ is a tiling of $P$, then $\{x_J\}_{J\in \vert(T)}$ is a transcendence basis for the coordinate ring of $\Yo$.  
\end{theorem}

\begin{proof}
Fix a tiling $T$, and let
$f_T^I(x_J)\in\QQ[x_J^{\pm1}]_{J\in\vert(T)}$ denote the Laurent polynomials expressing $x_I$ in terms of the variables $\{x_J\}_{J\in\vert(T)}$.
The existence of those Laurent polynomials is guaranteed by Proposition~\ref{yo}.
%They have the property that, given any collection $u_J\in\RR_{>0}$, $J\in\vert(T)$, of positive real numbers, the values $f_T^I(u_J)$ are then also positive real.
Let 
\[
U:=\Spec\,\Big(\!\big(\CC[x_J^{\pm1}]_{J\in\vert(T)}\big)\,[f_T^I(x_J)^{-1}]_{I\in\Pi}\Big)
\]
be the open subvariety of $(\CC^\times)^{\vert(T)}$ obtained by removing the hypersurfaces $\{f_T^I=0\}$, for all $I \in \Pi$. 
%We know that $U$ is nonempty because it contains $(\RR_{>0})^{\vert(T)}$.

We claim that 
the projection $p:(\CC^\times)^\Pi\to(\CC^\times)^{\vert(T)}$ induces an isomorphism of $\Yo\simeq U$.
Indeed, the map
\[
\begin{split}
q:\big(\CC^\times\big)^{\vert(T)}\longrightarrow&\hspace{.4cm} \CC^\Pi\\
(x_J)\hspace{.4cm}\mapsto\,\,& \big(f_T^I(x_J)\big)_{I\in\Pi}
\end{split}
\]
maps $U$ into $(\CC^\times)^\Pi$, and the composite $p\circ q|_U$ is the inclusion of $U$ into $(\CC^\times)^{\vert(T)}$.
By Theorem~\ref{h:i}, the Laurent polynomials $f_T^I(x_J)$ satisfy the equations (\ref{trc}) defining $\Yo$.
It follows that $q(U)\subset \Yo$. 
We now show that $p(\Yo)\subset U$.
By definition of $f_T^I$, we have $x_I=f_T^I(x_J)$ on $\Yo$.
Since $\Yo\subset (\CC^\times)^\Pi$, the variables $x_I$ only take invertible values on $\Yo$.
It follows that $p(\Yo)$ doesn't hit the hypersurfaces $\{f_T^I=0\}$.
Finally, we check that $q\circ p|_{\put(1.5,6.3){$\scriptscriptstyle \circ$}Y}$ is the inclusion of $\Yo$ into $\CC^\Pi$.
This is indeed the case since, on $\Yo$, the coordinates $\{x_I\}_{I\in\Pi\setminus\vert(T)}$ are entirely determined by 
the formulas $x_I=f_T^I(x_J)$, which agrees with the definition of $q$.

So we have shown that $\Yo$ is isomorphic to $U$. 
As $U$ is an open subvariety of $(\CC^\times)^{\vert(T)}$, it is an irreducible variety of dimension $\# T$.
By Lemma~\ref{fw}, $\# T=\sum_{ i < j } a_i a_j+\sum_i a_i+1$, so $U$ is irreducible of this dimension.
The set $\{x_J\}_{J\in\vert(T)}$ is then clearly a transcendence basis for its coordinate ring.
%
%Given lemma~\ref{Unordered}, the proof that $U \isomorph \Yo(A)$ is quite straightforward. Consider the map $p: \Yo(A) \to (\CC^*)^{\vert(T)}$ given by projection onto the $\vert(T)$ coordinates. By lemma~\ref{Unordered}, the image lands in $U$ and the map $\Yo(A) \to U$ is bijective. To see that this bijective map is an isomorphism, we note that, if $R$ is any $\CC$-algebra, then Lemma~\ref{Unordered} also shows that $p$ induces a bijection from the $\Spec R$ valued points of $\Yo(A)$ to the $\Spec R$ valued points of $U$.  We see from Yoneda's lemma that the map is an isomorphism of affine schemes.
%
%Since every coordinate function on $\Yo(A)$ is given by a rational expression in the variables $X_I$, $I \in \vert(T)$, we see that these variables are a transcendence basis for (and, in fact, generate) the fraction field of $\OO(\Yo(A))$. 
%Finally, $X(A)$ is a quotient of $\Yo(A)$ by a free action of $(\CC^*)^2$, and thus has dimension two less than that of $\Yo(A)$.
\end{proof}

\begin{corollary}
$X(A)$ is an irreducible variety of dimension $\sum_{ i < j } a_i a_j+\sum_i a_i-1$. \qed
\end{corollary} 

\begin{proof}
We showed above that $\Yo(A)$ is irreducible of dimension $\sum_{ i < j } a_i a_j+\sum_i a_i+1$.
So $\Xo(A)$, the quotient of $\Yo(A)$ by a free action of $(\CC^{*})^2$, is irreducible of dimension $\sum_{ i < j } a_i a_j+\sum_i a_i-1$.
$X(A)$ is the closure of $\Xo(A)$, so it is also irreducible of dimension $\sum_{ i < j } a_i a_j+\sum_i a_i-1$.
\end{proof}

\section{The Isotropic Grassmannian}

In this section, we set $A=(1,1,\ldots,1)$.
Our goal is to prove that the variety $X$ defined in the introduction is then isomorphic to $\IG(n-1,2n)$,
the variety of isotropic $n-1$ planes in $\CC^{2n}$.
(For $V$ is a vector space with a symmetric, nondegenerate bilinear form $\langle, \rangle$, a subspace $W$ of $V$ is called isotropic if $\langle w_1, w_2 \rangle=0$ for all $w_1$ and $w_2 \in W$.) 
We remark that the main complexities of this section arise from keeping careful track of signs. 
%the reader may be able to best follow the flow of our argument by imagining that we are working in a field of characteristic $2$.

It will be convenient to slightly modify our defining equations (\ref{trc}).
Let $\Yo'$ be the subvariety of $(\CC^\times)^{2^n}$ defined by \dontshow{trbi}
\begin{equation}\label{trbi}
  x_{I}\, x_{I+e_j+e_k+e_\ell} 
+ x_{I+e_j+e_\ell}\, x_{I+e_k}
= x_{I+e_j+e_k}\, x_{I+e_\ell} 
+ x_{I+e_k+e_\ell}\, x_{I+e_j} 
\end{equation}
for $j < k < \ell$.
As in the introduction, we let $\Xo'$ be the quotient
$\Yo'/(\CC^\times)^2$, and let
$X'\subset\CC\PP^{2^{n-1}-1}\times\CC\PP^{2^{n-1}-1}$ be its closure.

\begin{lemma}
$X'$ is isomorphic to $X$.
\end{lemma}

\begin{proof}
Let $\varphi(i_1,\ldots,i_n):=\sum i_j$.
The isomorphism is provided by the involution that
negates those variables $x_I$ for which 4 divides $\varphi(I)$.
\end{proof}

It will be useful to identify the two copies of $\CC\PP^{2^{n-1}-1}$ with the projectivizations of the spin representations of $\Spin(2n)$.
We begin with some standard background about the spin group.
Fix an $n$-dimensional vector space $W$ with a basis $\epsilon_1,\epsilon_2,\ldots, \epsilon_n$, and let $W^\vee$ be its dual, along with its dual basis $\epsilon_1^\vee,\ldots, \epsilon_n^\vee$.
Equip $V:=W\oplus W^\vee$ with the inner product
\[
\big\langle (w_1, w^{\vee}_1), (w_2, w^{\vee}_2) \big\rangle := {\textstyle\frac{1}{2}} \big( w_1^{\vee}(w_2) + w_2^{\vee}(w_1) \big).
\]
The Clifford algebra $\Cl(2n)$ is then defined to be the quotient of $\bigoplus_{k=0}^{\infty} V^{\otimes k}$ by $v \otimes v - \langle v, v \rangle$ or, 
equivalently, by $ v \otimes w + w \otimes v - 2 \langle v,w \rangle$. 
It decomposes as $\Cl(2n)^{\even} \oplus \Cl(2n)^{\odd}$, the images of $\bigoplus V^{\otimes 2k}$ and $\bigoplus V^{\otimes 2k+1}$ respectively. 
For $c  \in \Cl(2n)$, we write $(-1)^c$ to denote $1$ if $c$ is even and $-1$ if $c$ is odd.

If $v \in V$ has length $1$ then $v^2=1$ in $\Cl(2n)$ and, in particular, $v$ is a unit. 
We define $\Pin(2n)$ to be the subgroup of the unit group of $\Cl(2n)$ generated by such $v$. 
Every element of $\Pin(2n)$ is either odd or even; we write $\Spin(2n)$ for the subgroup of even elements. 
For $v \in V$ of length $1$, the map $\rho(v) : \Cl(2n)\to\Cl(2n)$ given by $x \mapsto -v xv$ takes $V$ to itself. 
Specifically, acts on $V$ by reflection in the hyperplane orthogonal to $v$. 
Thus, $\rho(\gamma) : x \mapsto (-1)^{\gamma} \gamma x \gamma^{-1}$ provides an action of $\Pin(2n)$ on $V$. 
Since this action is generated by orthogonal reflections, we obtain natural maps $\Pin(2n) \to O(2n)$ and $\Spin(2n) \to \SO(2n)$; these maps are double covers. 

Let $S:=\bwedge^{\bullet}(W)$ denote the exterior algebra of $W$. 
We make $S$ into a $\Cl(2n)$ module by defining 
\[
\begin{split}
u \cdot (w_1 \wedge \cdots \wedge w_k)&:=u \wedge w_1 \wedge \cdots \wedge w_k\hspace{1cm} \text{for}\quad u \in W\\
u^{\vee} \cdot (w_1 \wedge \cdots \wedge w_k) &:= u^\vee \inn (w_1 \wedge \cdots \wedge w_k) \quad \text{for}\quad u^\vee \in W^\vee,
\end{split}
\]
where
$u^\vee \inn (w_1 \wedge \cdots \wedge w_k)=\sum (-1)^{j-1} u^{\vee}(w_j)\left(  w_1 \wedge \cdots \wedge \widehat{w_j}  \wedge \cdots w_k \right)$, 
and the hat means that we omit the $j^{\textrm{th}}$ term. 
Then $S$ is a $\Spin(2n)$ representation by restriction to $\Spin(2n) \subset \Cl(V)$. 
As a $\Spin(2n)$ representation, it splits into two summands
\[
S_+:=\bwedge^{\even} W,\qquad\text{and}\qquad S_-:=\bwedge^{\odd} W. 
\]
Recall that in this section we have $\Pi=\{0,1\}^n$. 
For $J \in \Pi$, we let $v_J:=\epsilon_{j_1} \wedge \cdots \wedge \epsilon_{j_r}$, where $j_1 < \cdots < j_r$ are the indices for which $J_{j}=1$. 
Then $\{v_J\}_{J\in\Pi}$ forms a basis of $S$. 
Similarly, $\{v_J\}_{J\in\Pi^\even}$ and $\{v_J\}_{J\in\Pi^\odd}$ form bases of $S_+$ and $S_-$ respectively.

We now describe how $\IG(n-1,2n)$ sits inside $\PP(S_+) \times \PP(S_-)$. 
Our result is then that, using the above bases to identify $\PP^{2^{n-1}-1} \times \PP^{2^{n-1}-1}$ with $\PP(S_+) \times \PP(S_-)$,
the variety $\IG(n-1,2n)$ is equal to $X'$.
We first need a computation in the representation theory of $\Spin(2n)$.

\begin{Proposition}\label{iad}\dontshow{iad}
$\bwedge^{n-1} V$ is a direct summand of $S_+ \otimes S_-$. \qed
\end{Proposition}

\begin{proof}
Let $s_0^+$ and $s_0^-$ be high weight vectors for $S_+$ and $S_-$. 
We write weights for $\Spin(2n)$ as $n$-tuples $(w_1, \ldots, w_n)$ with $w_i$ all odd or all even. 
So $s_0^+$ and $s_0^-$ have weights $(1,1,\ldots,1,1)$ and $(1,1,\ldots,1,-1)$ respectively.
Then the vector $s_0^+ \otimes s_0^-$, in $S_+ \otimes S_-$, is a high weight vector of weight $(2,2,\ldots,2,0)$. 
Since $\bwedge^{n-1} V$ is the simple representation with high weight vector of weight $(2,2,\ldots,2,0)$, we deduce that $\bwedge^{n-1} V$ is a summand of $S_+ \otimes S_-$.
\end{proof}

Continue the notations $s_0^+$ and $s_0^-$, and set $v_0=s_0^+ \otimes s_0^-$. 
% Concretely, we may pick them to be $v_0:=e_1\wedge\ldots\wedge e_{n-1}$, and 
% \[
% \begin{array}{lll}
% s_0^+:=e_1\wedge\ldots\wedge e_n,& s_0^-:=e_1\wedge\ldots\wedge e_{n-1}& \text{if\, $n$\, is even},\\
% s_0^+:=e_1\wedge\ldots\wedge e_{n-1},& s_0^-:=e_1\wedge\ldots\wedge e_n& \text{if\, $n$\, is odd}.
% \end{array}
% \]
The inclusion $\PP(\bwedge^{n-1} V)\hookrightarrow \PP(S_+ \otimes S_-)$ sends $[v_0]$ to $[s_0^+ \otimes s_0^-]$.
Let $L\subset V$ denote the span of $\epsilon_1,\ldots,\epsilon_{n-1}$.
Then $L$ is an $n-1$ dimensional isotropic subspace, and its image
under the Pl\"ucker embedding $\G(n-1,2n)\hookrightarrow \PP(\bwedge^{n-1} V)$ is exactly $[v_0]$.
Here, $\G(n-1,2n)$ denotes the Grassmannian of all $n-1$ dimensional subspaces of the $2n$ dimensional space $V$.
Now consider the Segre embedding $\PP(S_+) \times \PP(S_-) \into \PP(S_+ \otimes S_-)$.
Since $\IG(n-1,2n)$ consists of a single $\Spin(2n)$ orbit, and $\PP(S_+) \times \PP(S_-)$ is a $\Spin(2n)$ invariant subscheme of $\PP(S_+ \otimes S_-)$,
the image of $\IG(n-1,2n)$ in $\PP(S_+ \otimes S_-)$ lands inside $\PP(S_+) \times \PP(S_-)$.
So we get the following diagram of inclusions: 
\begin{equation}
\label{cem}
\phantom{.}
\put(61,35){\line(1,0){300}}
\put(61,35){\line(-2,-1){10}}
\put(44.6,27.8){\rotatebox{30}{$\scriptscriptstyle \subset$}}
\put(358.9,32.9){\rotatebox{-30}{$\to$}}
\begin{array}{ccccccccc}
\!\IG(n\!-\!1,2n)\!&\!\!\!\hookrightarrow\!\!\! &
\!\G(n\!-\!1,2n)\!&\!\!\!\hookrightarrow\!\!\! &
\!\PP(\bwedge^{n-1} V)\!&\!\!\!\hookrightarrow\!\!\! &
\!\PP(S_+ \!\otimes\! S_-)\!&\!\!\!\hookleftarrow\!\!\! &
\!\PP(S_+) \!\times\! \PP(S_-)\! \\
\rotatebox{90}{$\in$}&&
\rotatebox{90}{$\in$}&&
\rotatebox{90}{$\in$}&&
\rotatebox{90}{$\in$}&&
\rotatebox{90}{$\in$}
\\
L&\mapsto&L&\mapsto&[v_0]&\mapsto&[s_0^+ \otimes s_0^-]&\raisebox{1.05ex}{\rotatebox{180}{$\mapsto$}}&([s_0^+], [s_0^-])
\end{array}
\end{equation}
\dontshow{cem} We can now state the main theorem of this section.

\begin{theorem}\label{Ttm}\dontshow{Ttm}
Using the identification of $\PP(\CC^{\Pi^\even})\times\PP(\CC^{\Pi^\odd})$
with $\PP(S_+)\times \PP(S_-)$
provided by the bases $\{v_J\}_{J\in\Pi^\even}$ of $S_+$ and $\{v_J\}_{J\in\Pi^\odd}$ of $S_{-}$, 
the subvariety $X'$ defined in (\ref{trbi})
is equal to the image of $\IG(n-1,2n)$ under the embedding (\ref{cem}).
\end{theorem}

%NEW STUFF HERE

Instead of the representation theory above, we could describe the coordinates on $\PP(S_+) \times \PP(S_-)$ explicitly in terms of Pfaffians. 
We sketch this approach now, but we will not use it.
The isotropic Grassmannian $\IG(n,2n)$ has two connected components, which we will denote $\LG_+(2n)$ and $\LG_-(2n)$.
Given an isotropic subspace $K \subset V$, of dimension $n-1$, there are precisely two $n$-dimensional isotropic subspaces containing $K$.
One of these subspaces, which we will call $L_+$, corresponds to a point of $\LG_{+}(2n)$, and the other, which we will call $L_-$, to a point of $\LG_{-}(2n)$
Also, $K$ can be recovered from $L_{+}$ and $L_{-}$ as $L_{+} \cap L_{-}$.
Thus, $\IG(n-1, 2n)$ embeds in $\LG_{+}(2n) \times \LG_{-}(2n)$.

The Lagrangian Grassmannians $\LG_{\pm}(2n)$ embed in $\PP(S_{\pm})$ and our embedding of $\IG(n-1, 2n)$ can be described as the composition $\IG(n-1, 2n) \into \LG_{+}(2n) \times \LG_{-}(2n) \into \PP(S_+) \times \PP(S_-)$. 
So the coordinates corresponding to the even and odd vertices of the cube are coordinates of the subspaces $L_{+}$ and $L_{-}$.
We now discuss the meaning of these coordinates.
Let $J$ be an element of $\Pi^{\even}$, let $q_J$ be the corresponding coordinate on $\PP(S_+)$, and let $I$ be the corresponding subset of $[n]$.

A generic point of $\LG_{+}(2n)$ is the rowspan of a $n \times (2n)$ matrix of the form $\left( \begin{smallmatrix} A & \Id \end{smallmatrix} \right)$, where $A$ is a skew-symmetric $n \times n$ matrix. 
Let $A_{II}$ be the submatrix of $A$ consisting of the rows and columns indexed by $I$; then $q_J(L_{+})$ is the Pfaffian of $A_{II}$. 
The case of $\LG_{-}$ is similar but a bit messier. 
This concludes our sketch of the approach using Pfaffians.

We will soon verify the case $n=3$ of Theorem~\ref{Ttm}.
When $n$ is odd, $S_+$ and $S_-$ are dual representations of the spin group.
It will be essential to know the explicit pairing between these spaces, in our preferred bases. 

\begin{lemma} \label{pairing3}
Let $n=3$.
Then the bilinear pairing $B_3:S_+\times S_-\to \CC$ given by
%Let $B_3( \ , \ )$ be the equivariant pairing between $S_+$ and $S_-$, which is unique up to scalar multiplication. Then
\begin{align*}
B_3(v_{000}, v_{111}) &= 1 & B_3(v_{011}, v_{100}) &= -1 \\
B_3(v_{101}, v_{010}) &= 1 & B_3(v_{110}, v_{001}) &= -1
\end{align*}
is $\Spin(3)$ invariant.
\end{lemma}

In order to prove Lemma~\ref{pairing3}, we describe the invariant bilinear form for general $n$.
Let $\Vol: \bwedge^n W \to \CC$ denote the standard volume form sending $\epsilon_1\wedge\ldots\wedge \epsilon_n$ to 1,
and let $\langle \ , \ \rangle_k: \bwedge^k W \times \bwedge^{n-k} W \to \CC$ be the pairing $\langle \eta, \omega \rangle_k:=\Vol(\eta \wedge \omega)$.
Finally, let $B( \ , \ )$ be the bilinear form on $S=\bwedge^{\bullet} W$ given by $\sum_{k=0}^n (-1)^{k(k-1)/2} \langle \ , \ \rangle_k$. 
The form $B$ is non-degenerate and, depending on the parity of $\lfloor \frac{n}{2}\rfloor$, either symmetric or antisymmetric.

\begin{lemma}
The form $B$ is invariant under the action of $\Pin(2n)$.
\end{lemma}

This is a routine computation; since $\Pin(2n)$ is generated by vectors $v\in V$ of length 1,
it is enough to check that
\[
B(v \cdot s_1, v \cdot s_2)=B(s_1, s_2)
\]
for such $v\in V$, and vectors $s_1, s_2\in S$. We leave this to the reader.

\begin{proof}[Proof of Lemma~\ref{pairing3}]
We compute:
\begin{alignat*}{3}
B_3(v_{000}, v_{111}) &= \phantom{-} \Vol \left( 1 \wedge ( \epsilon_1 \wedge \epsilon_2 \wedge \epsilon_3 ) \right) \ &= \phantom{-}1 \\
B_3(v_{011}, v_{100}) &= - \Vol \left( (\epsilon_2 \wedge \epsilon_3) \wedge  \epsilon_1  \right) &= -1 \\
B_3(v_{101}, v_{010}) &= - \Vol \left( (\epsilon_1 \wedge \epsilon_3) \wedge  \epsilon_2  \right) &= \phantom{-}1 \\
B_3(v_{110}, v_{001}) &= - \Vol \left( (\epsilon_1 \wedge \epsilon_2) \wedge  \epsilon_3  \right) &= -1 
\end{alignat*}
\end{proof}

We now verify the case $n=3$ of Theorem~\ref{Ttm}.
Let $S_+(6)$ and $S_-(6)$ be the two spin representations of $\Spin(6)$,
and let $x_{000},x_{110},x_{101},x_{011}\in S_+(6)^\vee$, and  
$x_{100},x_{010},x_{001},x_{111}\in S_-(6)^\vee$ be the bases dual to $\{v_J\}_{J\in\Pi^\even}$ and $\{v_J\}_{J\in\Pi^\odd}$.

\begin{lemma}\label{pxi}\dontshow{pxi}
The subvariety of $\PP(S_+(6))\times\PP(S_-(6))$ defined by the equation
\begin{equation}\label{los3}
x_{000}\, x_{111} +
x_{101}\, x_{010} =
x_{110}\, x_{001}
+ x_{011}\, x_{100}
\end{equation}
is equal to $\IG(2,6)$ in its embedding (\ref{cem}).
\end{lemma}

\begin{proof}
Note that $\IG(2,6)$ has dimension $5$, so it is a hypersurface in the $6$-dimensional variety $\PP(S_+) \times \PP(S_1)$.
So it is enough to check that this equation does vanish on $\IG(2,6)$.

Notice that $\bigwedge^2 V$ has dimension $15$, while $S_+ \otimes S_-$ has dimension $16$. 
The kernel of the pairing $B_3 : S_+ \otimes S_- \to \CC$ is a $15$-dimensional sub-representation of $S_+ \otimes S_-$, so it must be equal to $\bigwedge^2 V$. 
(We use the standard fact that $\bigwedge^2 V$ is irreducible.)
By Lemma~\ref{pairing3}, the equation of this kernel is $x_{000} \otimes x_{111} - x_{011} \otimes x_{100} + x_{101} \otimes x_{010} - x_{110} \otimes x_{001}$.
Restricted to $\PP(S_+) \times \PP(S_-)$, this equation becomes $x_{000}  x_{111} - x_{011}  x_{100} + x_{101} x_{010} - x_{110}  x_{001}$.
Since $\IG(2,6)$ is in $\bigwedge^2 V$, we see that this equation does vanish on $\IG(2,6)$, as required.
\end{proof}

We will use a similar approach in the general case: showing that $\IG(n-1, 2n)$ has the right dimension and then showing that it obeys the equations~(\ref{trbi}).
We perform the dimensional computation now.

\begin{lemma}\label{codi}\dontshow{codi}
Both $X'$ and $\IG(n-1,2n)$ are of dimension $\binom{n+1}{2}-1$.
\end{lemma}

\begin{proof}
By Theorem~\ref{irred}, the variety $\Yo$ is irreducible of dimension $\binom{n}{2}+n+1=\binom{n+1}{2}+1$.
It follows that $\Xo'\simeq \Xo=\Yo/(\CC^\times)^2$ is of dimension $\binom{n+1}{2}-1$.
Since $\Xo'$ is dense in $X'$, they have the same dimension.

The dimension of $\IG(n-1,2n)$ is a standard computation; the isotropic Grassmannian is a homogeneous space for $\SO(2n)$, which has dimension $\binom{2n}{2}$, and the stabilizer of any point has dimension $3\binom{n}{2}+1$.
\end{proof}

We now begin to show that $\IG(n-1,2n)$ obeys the equations~(\ref{trbi}).
Let $j<k<\ell$ be indices between $1$ and $n$, and let $I\in \Pi$ be a vertex with $i_j=i_k=i_\ell=0$.
Let $\Spin^{\{jk\ell\}}(6)\subset\Spin(2n)$ denote the copy of $\Spin(6)$ corresponding to the coordinates $j,k,\ell$. 
The key point will be to construct a projection $\pi_{\otimes} : S_+ \otimes S_- \to S_+(6) \otimes S_-(6)$, while keeping careful track of signs. 

Set $v^\vee_I:=\epsilon^\vee_{i_r}\wedge \epsilon^\vee_{i_{r-1}}\wedge\ldots\wedge \epsilon^\vee_{i_1}$, where $i_1<\ldots<i_r$ are the indices such that $I_i=1$.
Write $p$ for the projection from $S$ onto the $\Spin^{\{jk\ell\}}(6)$ subrepresentation spanned by $\bigwedge^{\bullet} \Span(\epsilon_j, \epsilon_k, \epsilon_{\ell})$.
Then set $\pi(v) := p(v^\vee_I\inn s)$.
The restrictions of $\pi$ to $S_+$ and  $S_-$ then provide maps
\[
\begin{split}
\pi_+:S_+\to S_+^{\{jk\ell\}}(6),\qquad \pi_-:S_-\to S_-^{\{jk\ell\}}(6)\qquad \text{if}\,\, I \in  \Pi^\even,\\
\pi_+:S_-\to S_+^{\{jk\ell\}}(6),\qquad \pi_-:S_+\to S_-^{\{jk\ell\}}(6)\qquad \text{if}\,\, I \in  \Pi^\odd.
\end{split}
\]
It is easy to check that $\pi_+$ and $\pi_-$ are $\Spin^{\{jk\ell\}}(6)$ equivariant.

Our main difficulty is that $\pi_-(v_{I+e_j})$ may be either $v_{100}$ or $-v_{100}$, depending on the choice of $I$, and similarly for the other seven coordinates.
However, our sign difficulties cancel when we tensor $\pi_+$ and $\pi_-$ together.
Let $a$, $b$, $c$, $d$ be the number of nonzero coordinates of $I$ in $[1,j)$, $(j,k)$, $(k,\ell)$ and $(\ell, n]$ respectively.
Then careful compuation gives the following lemma:

\begin{lemma}\label{qvs}\dontshow{qvs}
We have
\begin{alignat*}{10}
&\pi_{+}(v_I) &=& & v_{000} & \quad \quad &\pi_{-}(v_{I+ e_j+ e_k+ e_{\ell}}) &=& (-1)^{b+d} & v_{111} \\
&\pi_{+}(v_{I+ e_k+ e_{\ell}}) &=& (-1)^{c} & v_{011} & \quad \quad &\pi_{-}(v_{I+ e_j}) &=& (-1)^{b+c+d} & v_{100} \\ 
&\pi_{+}(v_{I+ e_j+ e_{\ell}}) &=& (-1)^{b+c} & v_{101} & \quad \quad &\pi_{-}(v_{I+ e_k}) &=& (-1)^{c+d} & v_{010} \\ 
&\pi_{+}(v_{I+ e_j+ e_{k}}) &=& (-1)^{b} & v_{110} & \quad \quad &\pi_{-}(v_{I+ e_{\ell}}) &=& (-1)^{d} & v_{001} \\ 
\end{alignat*}
and therefore
\begin{align*}
(\pi_{+} \otimes \pi_{-}) (v_{I}\,{\otimes}\, v_{I+ e_j+ e_k+ e_\ell}) &= (-1)^{b+d} v_{000}\,{\otimes}\, v_{111} \\
(\pi_{+} \otimes \pi_{-})(v_{I+ e_j+ e_\ell}\,{\otimes}\, v_{I+ e_k}) &= (-1)^{b+d} v_{101}\,{\otimes}\, v_{010}  \\
(\pi_{+} \otimes \pi_{-})(v_{I+ e_j+ e_k}\,{\otimes}\, v_{I+ e_\ell}) &=  (-1)^{b+d} v_{110}\,{\otimes}\, v_{001} \\
(\pi_{+} \otimes \pi_{-})(v_{I+ e_k+ e_\ell}\,{\otimes}\, v_{I+ e_j}) &= (-1)^{b+d} v_{011}\,{\otimes}\, v_{100}.
\end{align*}
\end{lemma}

Define $\tilde{\pi}_{\otimes}$ by $\pi_+ \otimes \pi_-$ when $I \in \Pi^{\even}$ and by $(\pi_+ \otimes \pi_-) \circ \mathrm{flip}$ when $I \in \Pi^{\odd}$.
Since $\pi_+$ and $\pi_-$ are $S^{jk \ell}(6)$ equivariant, so is $\tilde{\pi}_{\otimes}$.
Set $\pi_{\otimes} = (-1)^{b+d} \tilde{\pi}_{\otimes}$.
We summarize our above computations:

\begin{Proposition}  \label{pisummary}
 The map $\pi_{\otimes} : S_+ \otimes S_- \to S_+(6) \otimes S_-(6)$ is a $\Spin^{\{jk \ell\}}(6)$ equivariant map.
When $I \in \Pi^{\even}$, this maps satisfies
\begin{align*}
\pi_{\otimes}(v_I \otimes v_{I +  e_j +  e_k +  e_{\ell}}) &= v_{000} \otimes v_{111} & \pi_{\otimes}(v_{I+ e_j} \otimes v_{I +  e_k +  e_{\ell}}) &= v_{100} \otimes v_{011} \\ 
\pi_{\otimes}(v_{I+ e_k} \otimes v_{I +  e_j +  e_{\ell}}) &= v_{010} \otimes v_{101} & \pi_{\otimes}(v_{I+ e_{\ell}} \otimes v_{I +  e_j +  e_k}) &= v_{001} \otimes v_{110}  
\end{align*}
When $I \in \Pi^{\odd}$, the same relations hold except that the tensor factors on the left hand side must be switched (in order to put the even term first).
\end{Proposition}

We now show that $\IG(n-1,2n)$ obeys the equations~(\ref{trbi}).
Consider what $\pi_{\otimes}$ will do to $\bigwedge^{n-1}(V)$. 
The subspace of $\bigwedge^{n-1}(V)$ supported on the eight weights $I$, $I+ e_j$, etcetera is an irreducible representation of $S^{\{jk \ell\}}(6)$, isomorphic to $\bigwedge^2(\CC^6)$.
So the equivariant map $\pi_{\otimes}$ takes $\bigwedge^{n-1}(V)$ to the subspace $\bigwedge^2(\CC^6)$ of $S_+(6) \otimes S_-(6)$.
We know (Lemma~\ref{pxi}) that this subspace obeys the linear equation $x_{000} \otimes x_{111} - x_{100} \otimes x_{011} + x_{010} \otimes x_{101} - x_{001} \otimes x_{110} = 0$. 
So Proposition~\ref{pisummary} shows that $\bigwedge^{n-1}(V)$ obeys $x_{I} \otimes x_{I +  e_j +  e_k +  e_{\ell}} - x_{I+ e_j} \otimes x_{I +  e_k +  e_{\ell}} + x_{I+ e_k} \otimes x_{I +  e_j +  e_{\ell}} - x_{I+ e_{\ell}} \otimes x_{I +  e_j +  e_{k}}$.
(If $I \in \Pi^{\odd}$, then we should switch the order of the tensor factors so that the even term comes first.)
Since $\IG(n-1, 2n)$ is in $\PP(\bigwedge^{n-1}(V))$, we see that $\IG(n-1,2n)$ obeys~(\ref{trbi}).

We now complete the proof of Theorem~\ref{Ttm}. 
By Theorem~\ref{irred} and Lemma~\ref{codi}, the varieties $\Xo'$ and $\IG(n-1,2n)\cap (\CC^\times)^\Pi$
are irreducible and of the same dimension.
We have just computed that $\IG(n-1,2n)\cap (\CC^\times)^\Pi$ obeys the defining equations of $\Xo'$.
So $\IG(n-1,2n)\cap (\CC^\times)^\Pi = \Xo'$
Then $X'$, which is defined as the closure of $\Xo'$, must also be the closure of $\IG(n-1,2n)\cap (\CC^\times)^\Pi$.
This closure is $\IG(n-1,2n)$, completing the proof of Theorem~\ref{Ttm}.

\section{The tropical cube recurrence}

The tropical version of the recurrence (\ref{trc}) is obtained by
replacing the operations plus and times by max and plus respectively.
Namely, we get \dontshow{trp}
\begin{equation}\begin{split}\label{trp}
x_{I+e_j+e_\ell}+ x_{I+e_k} =  \max\big(&x_I + x_{I+e_j+e_k+e_\ell},\,\, x_{I+e_j+e_k} + x_{I+e_\ell},\,\, x_{I+e_k+e_\ell}+ x_{I+e_j}\big) 
\end{split}\end{equation}
where again we assume that $1 \leq j<k<\ell \leq n$.
Since plus distributes over max, and since all the formulas encountered in its proof are subtraction free, 
Theorem~\ref{h:i}  extends without problem to the tropical situation. We see that the set of solutions to the tropical cube recurrence is a polyhedral fan of dimension $\sum_{i<j} a_i a_j + \sum_i a_i +1$ which, for every tiling $T$, has a parametrization by $\RR^{\vert(T)}$ given by continuous piecewise linear functions. 

But the tropical recurrence exhibits a new feature:
certain inequalities are propagated by the recurrence.
Recall that $C=\prod_{i=1}^n[0,a_i]$.
Let $W$ be a hyperplane \dontshow{rW}
\begin{equation}\label{rW}
W=\{(i_1,\ldots,i_n)\in C\,|\,i_s=c\},
\end{equation}
where $1 \leq s \leq n$ and $1 \leq c \leq a_s-1$.

Let $\{x_I \}_{I\in \Pi}$ be a collection of numbers satisfying the recurrence (\ref{trp}).
Given an edge $\alpha=(I,I+e_i)$ in $W$, we say that the $x_I $ satisfy the $W$-inequalities at $\alpha$ if \dontshow{suk}
\begin{equation}\label{suk}
\begin{split}
x_I +x_{I+e_i}&\ge x_{I+e_s}+x_{I+e_i-e_s},\\
x_I +x_{I+e_i}&\ge x_{I-e_s}+x_{I+e_i+e_s}.
\end{split}
\end{equation}

Let $\Delta:=\pi^{-1}(\partial P)$ be the common boundary of all the tilings of $P$.
A {\em cutcurve} for $W$ is a path $\gamma=(\gamma(0),\gamma(1),\ldots)$ that connects the only two points 
$(a_1, \dots, a_{s-1}, c, 0, \dots, 0)$, $(0, \dots, 0, c, a_{s+1}, \dots, a_n)$
of $\Delta\cap W$, and such that 
$\gamma(t)-\gamma(t-1)$ only takes the values
$e_{s+1}$, $e_{s+2},\ldots,e_n$, $-e_1$, $-e_2,\ldots,-e_{s-1}$.
In other words, $\gamma$ is a geodesic for the taxi-cab metric on $C$. 
Geometrically, if $\gamma$ is a cutcurve contained among the edges of a tiling $T$ of $P$, 
then $\pi(\gamma)$ lies between the $c^{\textrm{th}}$ and $(c+1)^{\textrm{st}}$ pseudoline in the $s^{\textrm{th}}$ direction.

\[
\psfrag{ABCDEFGHIJKLMNOPQRSTUVWXYZABCDEFGHIJKLMNOPQRSTUVWXYZABCDEFGHIJKLMNOPQRSTUVWXYZ}
{\parbox{10cm}{A cutcurve (in bold) and the two parallel \\ pseudo-lines (dashed) between which it lies.}}
\epsfig{file=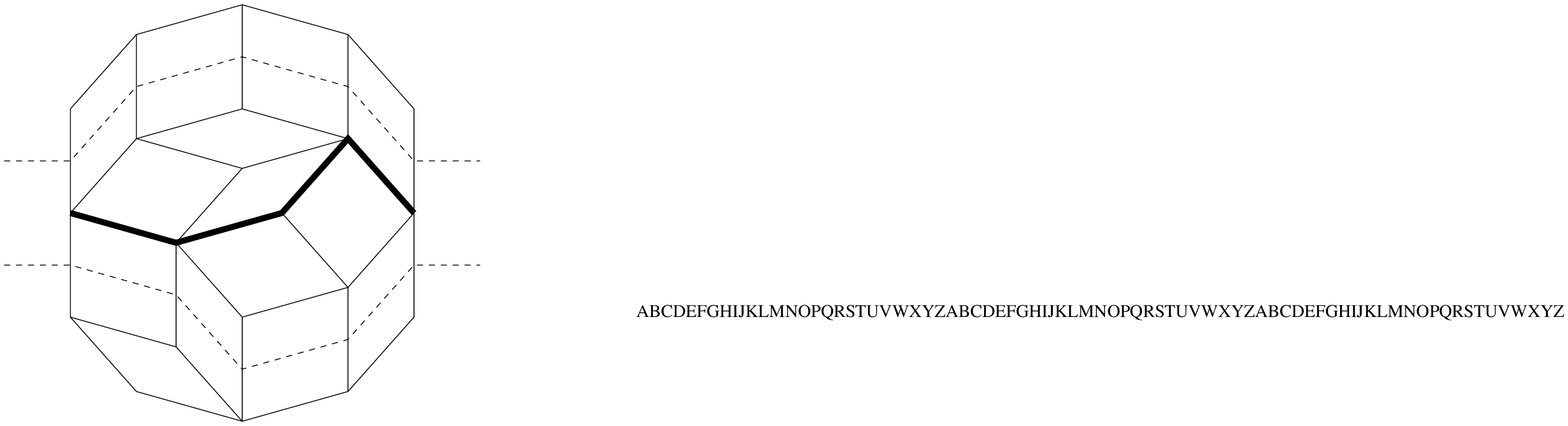, width=10 cm}
\]

%\begin{gather*}
%\epsfig{file=cutcurve.eps, width=3 cm} \\
%\parbox{10cm}{A cutcurve (in bold) and the two parallel \\ pseudo-lines (dashed) between which it lies}
%\end{gather*}

Let us call {\em elementary move} the operation of replacing a cutcurve $\gamma$ by another $\gamma'$, given by \dontshow{emm}
\begin{equation}\label{emm}
\gamma'(t)=\begin{cases}\gamma(t)&\text{if}\quad t\not= t_0\\
\gamma(t-1)+\gamma(t+1)-\gamma(t) &\text{if}\quad t= t_0,
\end{cases}
\end{equation}
where $t_0$ is an integer such that $\gamma(t_0-1)$, $\gamma(t_0)$, $\gamma(t_0+1)$ are not collinear.
An elementary move is like a small homotopy of $\gamma$ that goes over a square of $C(A)$.
It is an easy exercise to show that any two cutcurves are connected by a sequence of elementary moves.

\begin{Proposition}
Let  $\{x_I\}_{I\in \Pi}$ be a collection of numbers satisfying the tropical cube recurrence (\ref{trp}).
Let $W$ be as in (\ref{rW}), and let $\gamma$ be a cutcurve for $W$.
Suppose that the $W$-inequalities (\ref{suk}) are satisfied on all the edges of $\gamma$.
Then they are satisfied on all the edges of $W$.
\end{Proposition}

Henceforth, to emphasize the similarity between the tropical and the usual recurrences, 
we will use to the notations $x \odot y:=x+y$, $x \ooplus y:=\max(x,y)$, $x^{(-1)}:=-x$.

\begin{proof}
We show that for any cutcurve $\gamma'$, the $W$-inequalities are satisfied on the edges of $\gamma'$.
Since any two cutcurves can be joined by a sequence of elementary moves, it is enough to do the case when 
$\gamma$ and $\gamma'$ are separated by a single elementary move.

Let $t_0$ be as in (\ref{emm}), and let
$a=x_{\gamma(t_0-1)+e_s}$, 
$b=x_{\gamma(t_0)+e_s}$, 
$c=x_{\gamma(t_0+1)+e_s}$, 
$d=x_{\gamma'(t_0)+e_s}$, 
$p=x_{\gamma(t_0-1)}$, 
$q=x_{\gamma(t_0)}$, 
$r=x_{\gamma(t_0+1)}$, 
$s=x_{\gamma'(t_0)}$, 
$v=x_{\gamma(t_0-1)-e_s}$, 
$w=x_{\gamma(t_0)-e_s}$, 
$y=x_{\gamma(t_0+1)-e_s}$, 
$z=x_{\gamma'(t_0)-e_s}$.
These numbers satisfy \dontshow{mabc}
\begin{equation}\label{mabc}
w \odot s = v \odot r \ooplus z \odot q \ooplus y \odot p 
\quad \hbox{and} \quad q \odot d = p \odot c \ooplus s \odot b \ooplus r \odot a
\end{equation}
Their positions relative to $\gamma$ and $\gamma'$ is best understood via the following picture:

\vspace{.3cm}
\psfrag{a}{$a$}
\psfrag{b}{$d$}
\psfrag{c}{$c$}
\psfrag{d}{$b$}
\psfrag{p}{$p$}
\psfrag{q}{$s$}
\psfrag{r}{$r$}
\psfrag{s}{$q$}
\psfrag{v}{$v$}
\psfrag{w}{$z$}
\psfrag{y}{$y$}
\psfrag{z}{$w$}
\psfrag{gg}{$\gamma'$}
\psfrag{ggp}{$\gamma$}
\centerline{\epsfig{file=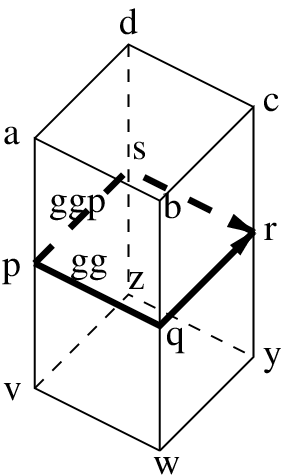, width=3cm}}
\vspace{.3cm}

We want to show that the $W$-inequalities on $\gamma$ and on $\gamma'$ are equivalent to each other.
The ones for $\gamma$ read:
\begin{eqnarray*}
       (A)\quad p \odot q \ge a \odot w,
& \qquad & (B) \quad p \odot q \ge b \odot v,\\
       (C)\quad q \odot r \ge c \odot w,
& \qquad & (D)\quad q \odot r \ge b \odot y,
\end{eqnarray*}
and the ones for $\gamma'$ read:
\begin{eqnarray*}
       (E)\quad p \odot s \ge a \odot z,
& \qquad & (F) \quad p \odot s \ge d \odot v,\\
       (G)\quad s \odot r \ge c \odot z,
& \qquad & (H)\quad s \odot r \ge d \odot y.
\end{eqnarray*}
It is enough to show that the inequalities for $\gamma$ imply those for $\gamma'$, as the conditions are symmetric in exchanging $\gamma$ and $\gamma'$. Also, by reflecting our picture over the $bqwdsz$ plane, 
it is enough to show that $(A)$, $(B)$, $(C)$ and $(D)$ together imply $(E)$ and $(F)$.

We begin by showing that $(A)$ implies $(E)$. We have
\[
\begin{split}
p \odot s  =  p \odot \big( 
v \odot r \ooplus  z \odot q \ooplus & y  \odot p \big)   \odot w^{(-1)} \\
& \geq   p \odot z \odot q \odot w^{(-1)} \geq a \odot z \odot w \odot w^{(-1)} = a \odot z
\end{split}
\]
where the second inequality is by $(A)$.

We now show that $(A)$, $(B)$ and $(C)$ imply $(F)$. We have
$$p \odot s = p \odot \big( v \odot r \ooplus z \odot q \ooplus y \odot p \big) \odot w^{(-1)} \geq p \odot v \odot r \odot w^{(-1)}.$$
Using $(A)$ and $(C)$ respectively, we have
\begin{equation}
p \odot s \geq  p \odot v \odot r \odot w^{(-1)} \geq a \odot v \odot r \odot q^{(-1)} \label{thing1}
\end{equation}
and
\begin{equation}
p \odot s \geq  p \odot v \odot r \odot w^{(-1)} \geq p \odot v \odot c \odot q^{(-1)}. \label{thing2}
\end{equation}
Also, by $(B)$, we have
\begin{equation} p \odot s \geq b \odot v \odot s \odot q^{(-1)} \label{thing3}. \end{equation}
``Adding" equations~(\ref{thing1}),~(\ref{thing2})  and~(\ref{thing3}) and using the relation $p \odot s \ooplus p \odot s=p \odot s$, we obtain
$$p \odot s \geq v \odot \left( a \odot r \ooplus p \odot c \ooplus b \odot s \right) \odot q^{(-1)} = v \odot d.$$
This is the desired relation $(F)$.
\end{proof}

%

%Let us also introduce the following inequalities as intermediary steps:
%\begin{equation*}
%\begin{split}
%       (I)\quad d+i\ge b+k,
%\qquad (J)\quad d+g+i\ge c+f+k,\\
%       (L)\quad d+i\ge c+j,
%\qquad (K)\quad d+g+i\ge b+g+j.
%\end{split}
%\end{equation*}
%As an easy consequence of (\ref{mabc}), we obtain the equivalences 
%$(B)\Leftrightarrow(F)(I)(J)$, $(D)\Leftrightarrow(H)(K)(L)$, $(E)\Leftrightarrow(A)(L)(I)$, and $(G)\Leftrightarrow(C)(J)(K)$.
%Putting all these together, we obtain the desired equivalence $(A)(B)(C)(D)\Leftrightarrow(E)(F)(G)(H)$, 
%as can be seen by following this diagram:
%\begin{equation}
%\begin{matrix}
%\xymatrix@R=.4cm@C=3cm{
%\,\phantom{(A)}\,&\ar@{}[dl]|{\displaystyle(L)}="LL"\,\phantom{(A)}\,\\
%(A)\ar@{-}[r]&\ar@{-}[dl]+<33pt,3pt>;+<-33pt,-3pt>|{\displaystyle(I)}\ar@{-} +<-33pt,3pt>;"LL"{\!\!\}}\Leftrightarrow(E)\hspace{.7cm}\\
%\hspace{.7cm}(B)\Leftrightarrow{\{\!\!}\ar@{-}[r]\ar@{-}[dr]+<-33pt,3pt>;+<33pt,-3pt>|{\displaystyle(J)}&(F)\\
%(C)\ar@{-}[r]&\ar@{-}[dl]+<33pt,3pt>;+<-33pt,-3pt>|{\displaystyle(K)}{\!\!\}}\Leftrightarrow(G)\hspace{.7cm}\\
%\hspace{.7cm}(D)\Leftrightarrow{\{\!\!}\ar@{}[dr]|{\displaystyle(L)}="L"\ar@{-} +<33pt,-3pt>;"L"\ar@{-}[r]&(H)\\
%\,\phantom{(A)}\,&\,\phantom{(A)}\,
%}\end{matrix}
%\end{equation}
%\end{proof}

\section{Speculations and Remarks}

We view this paper as an invitation. We have shown that the cube recurrence exhibits many of the combinatorial and algebraic features of the octahedron recurrence, yet we have not discovered where the cube recurrence comes from, nor do we imagine that we have found its most interesting properties. We close with some speculations regarding lines of research to pursue.

For an algebraic geometer of a classical inclination, a natural question is to recognize the varieties $X(A)$. We have shown that in some cases, these are isotropic Grassmannians. We suspect that they are always, in some way, related to the Lie groups of type $D$. Here are some questions to focus the investigation -- are these varieties smooth? Do they have any symmetries other than the obvious permutation and rescaling of coordinates?

In the case of the octahedron recurrence, the Laurentness property is a special case of a Laurentness property for cluster algebras. From the cluster algebra perspective, this generalization amounts to finding a recurrence defined on $N$-tuples of positive real numbers where, in each step of the recurrence, one replaces $(x_1, x_2, \dots, x_N)$ by $(f(x_2, \dots, x_N)/x_1, x_2, \dots, x_N)$, for some polynomial $f$. The cube recurrence only allows us to replace those variables which correspond to vertices of degree three in a tiling; if we are to discover something like the theory of cluster algebras, we should be able to replace any of the variables in such a manner. Our computation of the labels on the edges of the Caterpillar tree in Section~\ref{s:L} is, from this perspective, describing how to travel one step away from the moves in the cube recurrence. We pose the challenge of continuing to make many ``generalized flips" away from trivalent vertices. What is the rule that extends equation~(\ref{EdgeLabel})?

In the case of the octahedron recurrence, the fact that the tropical octahedron recurrence propagates certain inequalities allows one to use the tropical octahedron recurrence for computations with $\GL_n$ representations. (See~\cite{KTW} and~\cite{HK}.) Is there a similar connection between the tropical cube recurrence and representation theory? In a similar vein, the tropical octahedron recurrence has been shown in~\cite{HK} to be a disguised version of \emph{jeu d' taquin}. Is the tropical cube recurrence a disguised version of some combinatorial algorithm which is already known, or of one that should be?

Finally, to be extremely optimistic, one could try to generalize the results of~\cite{CarrSpey} and give a combinatorial formula for the Laurent polynomials produced by the multidimensional cube recurrence. This is probably extremely difficult because it should be harder than the corresponding problem for cluster algebras, which has been open for six years. Nonetheless, the authors, together with Dylan Thurston, have made some partial progress. A more tractable, still interesting question, might be to determine the Newton polytopes of these Laurent polynomials. 

\raggedright  
% WHAT DOES THIS COMMAND \raggedright DO?
% It tells LaTeX not to make the lines all the same length. This often makes URLs and journals with long titles look better.

\end{document}